\documentclass[12pt]{amsart}
\usepackage[top=1.08in, left=1.28in, bottom=1.08in, right=1.28in]{geometry}
\usepackage{amsmath,amsfonts,amsbsy,amsgen,amscd,mathrsfs,amssymb,amsthm,mathtools,bbm,enumerate}
\usepackage[colorlinks=true,citecolor=teal!80!black,linkcolor=blue]{hyperref}
\usepackage[foot]{amsaddr}

\usepackage{hhline}
\usepackage{fouriernc}

\usepackage{pgfplots}
\usepackage{tikz-cd}
\usetikzlibrary{arrows,automata}

\numberwithin{equation}{section}
\newtheorem{theorem}{Theorem}[section]
\newtheorem{corollary}[theorem]{Corollary}

\newtheorem{lemma}[theorem]{Lemma}
\newtheorem{proposition}[theorem]{Proposition}
\theoremstyle{definition}
\newtheorem{assumption}[theorem]{Assumption}
\newtheorem{definition}[theorem]{Definition}

\newtheorem{notation}[theorem]{Notation}

\newtheorem{openproblem}[theorem]{Open Problem}
\newtheorem{question}[theorem]{Question}
\newtheorem{remark}[theorem]{Remark}

\makeatletter\renewenvironment{proof}[1][\proofname] {\par\pushQED{\qed}\normalfont\topsep6\p@\@plus6\p@\relax\trivlist\item[\hskip\labelsep\bfseries#1\@addpunct{.}]\ignorespaces}{\popQED\endtrivlist}

\newcommand\al{\alpha}
\newcommand\be{\beta}
\newcommand\dd{\mathrm d}
\newcommand\De{\Delta}
\newcommand\de{\delta}
\newcommand\deq{\stackrel{\mathrm{distr.}}{=}}
\newcommand\eps{\varepsilon}
\newcommand\ga{\gamma}

\newcommand\ka{\kappa}

\newcommand\la{\lambda}
\newcommand{\om}{\omega}

\newcommand\Si{\Sigma}
\newcommand\si{\sigma}

\renewcommand\d{~\mathrm d}
\renewcommand\phi{\varphi}
\renewcommand\rho{\varrho}
\renewcommand\th{\vartheta}

\newcommand\mbb{\mathbb}
\newcommand\mbf{\mathbf}
\newcommand\mc{\mathcal}
\newcommand\mf{\mathfrak}
\newcommand\mr{\mathrm}

\newcommand\msf{\mathsf}

\begin{document}

\title[Spectral Geometry and Small Time Mass of 2D Anderson Models]
{On the Spectral Geometry and Small Time Mass of Anderson Models on Planar Domains}
\author{Pierre Yves Gaudreau Lamarre}
\address{Department of Mathematics,
Syracuse University,
Syracuse, NY 13244}
\email{pgaudrea@syr.edu}
\author{Yuanyuan Pan}
\email{ypan66@syr.edu}
\maketitle

\begin{abstract}
We consider the Anderson Hamiltonian (AH) and the parabolic Anderson model (PAM)
with white noise and Dirichlet boundary condition on a bounded planar domain $D\subset\mbb R^2$.
We compute the small time asymptotics of the AH's exponential trace up to order $O(\log t)$,
and of the PAM's mass up to order $O(t\log t)$.
Our proof is probabilistic, and relies on the asymptotics
of intersection local times of Brownian motions and bridges in $\mbb R^2$.
Applications of our main result include the following:

\noindent (i) If the boundary $\partial D$ is sufficiently regular, then $D$'s area and $\partial D$'s length
can both be recovered almost surely from a single observation of the AH's eigenvalues.
This extends Mouzard's Weyl law \cite{Mouzard} in the special case of bounded domains. 

\noindent (ii) If $D$ is simply connected and $\partial D$ is fractal, then $\partial D$'s Minkowski dimension (if it exists)
can be recovered almost surely from the PAM's small time asymptotics.

\noindent (iii) The variance of the white noise can be recovered almost surely from a single observation of the AH's eigenvalues.
\end{abstract}

\section{Introduction}

\subsection{The Anderson Models}

Let $D\subset\mbb R^2$ be bounded, nonempty, open, and connected.
Let $\ka\geq0$ be a nonnegative parameter, and let $\xi$ be a standard Gaussian white noise on $\mbb R^2$.
Informally, $\xi:\mbb R^2\to\mbb R$ is a centered Gaussian process with
\[\mbf E[\xi(x)\xi(y)]=\de_0(x-y),\qquad x,y\in \mbb R^2,\]
where $\de_0$ is the delta Dirac distribution. In rigorous terms, $\xi$ is
defined as a centered Gaussian process on $L^2(\mbb R^2)$ (interpreting $\xi(f)=\int f\xi$)
with covariance
\begin{align}
\label{Equation: Inner Product Covariance}
\mbf E[\xi(f)\xi(g)]=\langle f,g\rangle,\qquad f,g\in L^2(\mbb R^2),
\end{align}
where $\langle\cdot,\cdot\rangle$ denotes the standard inner product in $L^2(\mbb R^2)$.
We also use $\|\cdot\|$ to denote the standard norm in $L^2(\mbb R^2)$ throughout.

In this paper, we are interested in two well-studied Anderson models on $D$ with noise $\ka\xi$:
On the one hand, we consider the Anderson Hamiltonian (AH)
\begin{align}
\label{Equation: AH}
H_\ka=-\tfrac12\De+\ka\xi,
\end{align}
which acts on a dense subspace of $L^2(D)$ that satisfies the Dirichlet boundary condition
on $\partial D$. On the other hand, we consider the parabolic Anderson model (PAM)
$u_\ka(t,x)=\mr e^{-tH_\ka}\mbf 1_D(x)$, where $\mbf 1_{\{\cdot\}}$ denotes the
indicator function. In other words,
\begin{align}
\label{Equation: PAM}
\begin{cases}
\partial_tu_\ka(t,x)=\big(\tfrac12\De-\ka\xi(x)\big)u_\ka(t,x)&\qquad t>0,~x\in D,\\
u_\ka(0,x)=1&\qquad x\in D,\\
u_\ka(t,x)=0&\qquad t>0,~x\in\partial D.
\end{cases}
\end{align}

When $\ka=0$, \eqref{Equation: AH} and \eqref{Equation: PAM} respectively reduce
to the Dirichlet Laplacian and heat equation on $D$, whose definitions are well-known
and classical. However, when $\ka>0$, the AH and PAM are notoriously
difficult to construct rigorously due to the irregularity of $\xi$.
The standard approach to get around this obstacle is to introduce a sequence
of smooth approximations of the AH and PAM and take limits: Let
\begin{align}
\label{Equation: Gaussian Kernel}
p_t(x)=\frac{\mr e^{-|x|^2/2t}}{2\pi t},\qquad t>0,~x\in\mbb R^2
\end{align}
denote the planar Gaussian kernel. For every $\eps>0$, let
\begin{align}
\label{Equation: Smooth AH}
H_{\ka,\eps}=-\tfrac12\De+\ka\xi_\eps
\qquad\text{and}\qquad
u_{\ka,\eps}(t,x)=\mr e^{-tH_{\ka,\eps}}\mbf 1_D(x),
\end{align}
where $\xi_\eps:\mbb R^2\to\mbb R$ is a centered Gaussian
process with covariance
\begin{align}
\label{Equation: xi epsilon covariance}
\mbf E[\xi_\eps(x)\xi_\eps(y)]=p_\eps(x-y),\qquad x,y\in \mbb R^2.
\end{align}
Since $\xi_\eps$ has smooth
sample paths, we can define \eqref{Equation: Smooth AH}
using classical theory. Then, \eqref{Equation: AH}
and \eqref{Equation: PAM} can be constructed as the limits
\begin{align}
\label{Equation: Smooth AH Limits}
H_\ka=\lim_{\eps\to0}(H_{\ka,\eps}+\msf c_{\ka,\eps})
\qquad\text{and}\qquad
u_\ka(t,x)=\lim_{\eps\to0}u_{\ka,\eps}(t,x)\mr e^{-t\msf c_{\ka,\eps}},
\end{align}
where $\{\msf c_{\ka,\eps}:\eps>0\}$ are diverging renormalization constants that are designed to compensate
for $\xi_\eps$'s singularity as $\eps\to0$.
Constructions of this type have been carried out using regularity structures
\cite{Labbe,MatsudaVanZuijlen} and
paracontrolled calculus \cite{ChoukVanZuijlen,Mouzard}.
See also
\cite{AllezChouk,BailleulDangMouzard,DahlqvistDiehlDriver,GubinelliImkellerPerkowski,GubinelliUgurcanZachhuber,Hairer,HairerLabbe,MouzardOuhabaz}
for similar constructions when $D$ is replaced by a more general
two-dimensional manifold (with or without boundary) or $\mbb R^2$.
See Section \ref{Section: Construction of H and u} (more specifically, Assumption \ref{Assumption: AH}) for more details on the
assumptions we make in this paper regarding the construction of $H_\ka$ and $u_\ka$
via \eqref{Equation: Smooth AH Limits}.

\subsection{Main Result}

In this paper, we are interested in the following broad problem:

\begin{question}
\label{Question: Effect of Noise}
Let $\la_1(H_\ka)\leq\la_2(H_\ka)\leq\cdots$ be the eigenvalues of $H_\ka$,
and let $\psi_n(H_\ka)$ ($n\geq1$) be the corresponding orthonormal eigenfunctions.
How are $\la_n(H_\ka)$, $\psi_n(H_\ka),$
and the time-evolution of $u_\ka(t,\cdot)$ affected when one transitions
from $\ka=0$ to $\ka>0$?
\end{question}

See Sections \ref{Section: Past 1} and \ref{Section: Past 2} for a survey of past results
concerning Question \ref{Question: Effect of Noise}.

The approach that we adopt to study this question in this paper is based on
the exponential trace of $H_\ka$ and the mass of $u_\ka$, which
are respectively defined as
\begin{align}
\label{Equation: Exponential Trace}
\msf T_\ka(t)=\mr{Tr}\big[\mr e^{-tH_\ka}\big]=\sum_{n=1}^\infty\mr e^{-t\la_n(H_\ka)},\qquad t>0
\end{align}
and
\begin{align}
\label{Equation: Mass}
\msf M_\ka(t)=\int_Du_\ka(t,x)\d x=\sum_{n=1}^\infty\mr e^{-t\la_n(H_\ka)}\left\langle\psi_n(H_\ka),\mbf 1_D\right\rangle^2,\qquad t>0.
\end{align}
More specifically, we are interested in the small $t$ asymptotics of these quantities, which,
in view of the series expansions on the right-hand sides of \eqref{Equation: Exponential Trace}
and \eqref{Equation: Mass}, relate to the
large-$n$ asymptotics of $\la_n(H_\ka)$ and $\left\langle\psi_n(H_\ka),\mbf 1_D\right\rangle^2$.

When $\ka=0$, \eqref{Equation: Exponential Trace}
and \eqref{Equation: Mass} are typically called the heat trace and heat content
on $D$, respectively. The $t\to0$ asymptotics of $\msf T_0(t)$ and $\msf M_0(t)$ have been the subject
of considerable study in the past 60 years. In particular, it is known that these asymptotics contain a lot of information
about the geometry of $D$; see, e.g., \cite{Gilkey} and \cite[Section 1.2]{Rozanova-PierratTeplyaevWinterZahle} for more details and
references.
In this context, our main result quantifies the effect of $\ka$ on these quantities as follows:

\begin{notation}
Given a Borel measurable set $K\subset\mbb R^2$, we let $\msf A(K)=\int_K1\d x$ denote
its area/Lebesgue measure.
\end{notation}

\begin{theorem}
\label{Theorem: Main}
Let $\ka>0$ be fixed, and
suppose that $H_\ka$ and $u_\ka$ are constructed as described in
Assumption \ref{Assumption: AH}. As $t\to0$, it holds that
\begin{align}
\label{Equation: Main Asymptotic 1}
\mbf E\big[\msf T_\ka(t)\big]&=\msf T_0(t)+\tfrac{\ka^2\msf{A}(D)}{4\pi^2}\log t+o(\log t)
\qquad\text{and}\qquad
\mbf{Var}\big[\msf T_\ka(t)\big]=O(1),\\
\label{Equation: Main Asymptotic 2}
\mbf E\big[\msf M_\ka(t)\big]&=\msf M_0(t)+\tfrac{\ka^2\msf{A}(D)}{2\pi}t\log t+o(t\log t)
\qquad\text{and}\qquad\mbf{Var}\big[\msf M_\ka(t)\big]=O(t^2).
\end{align}
\end{theorem}

To the best of our knowledge, there are two past results that addressed this exact question
in the setting of planar domains:
Firstly, it was shown in \cite{Mouzard} that
\begin{align}
\label{Equation: Mouzard Law Intro}
\msf T_\ka(t)=\msf T_0(t)+o(t^{-1})=\tfrac{\msf{A}(D)}{2\pi}t^{-1}+o(t^{-1})\qquad\text{as }t\to0
\end{align}
when $\partial D$ is smooth. Then, \cite{MatsudaVanZuijlen} extended
\eqref{Equation: Mouzard Law Intro} to the case where $\partial D$ is Lipschitz.
See Section \ref{Section: Past 2} for more details,
including various extensions of \eqref{Equation: Mouzard Law Intro}
that exceed the setting of this paper (e.g., more general domains/manifolds
or boundary conditions).
In this context,
Theorem \ref{Theorem: Main} can be seen as a substantial refinement
of the estimate \eqref{Equation: Mouzard Law Intro}
for arbitrary planar domains. That is,
\eqref{Equation: Main Asymptotic 1} and \eqref{Equation: Main Asymptotic 2}
essentially state that
\begin{align}
\label{Equation: Informal Remainder Term 1}
\msf T_\ka(t)&=\msf T_0(t)+\tfrac{\ka^2\msf{A}(D)}{4\pi^2}\log t+o(\log t)&\text{as }t\to0,\\
\label{Equation: Informal Remainder Term 2}
\msf M_\ka(t)&=\msf M_0(t)+\tfrac{\ka^2\msf{A}(D)}{2\pi}t\log t+o(t\log t)&\text{as }t\to0,
\end{align}
with the important caveat that there may be some random fluctuations of
average size $O(1)$ and $O(t)$ in \eqref{Equation: Informal Remainder Term 1}
and \eqref{Equation: Informal Remainder Term 2} respectively.

\subsection{Outline of Applications}

Among other things, the asymptotics in \eqref{Equation: Informal Remainder Term 1}
and \eqref{Equation: Informal Remainder Term 2} suggest the following informal principle:
\begin{quote}
{\it"Any feature of $D$ that can be recovered from
the $t\to0$ asymptotics of $\msf T_0(t)$ and $\msf M_0(t)$ up to respective orders $O(\log t)$ and $O(t\log t)$
can also be recovered almost surely from $\msf T_\ka(t)$ and $\msf M_\ka(t)$ when $\ka>0$."}
\end{quote}
Using Theorem \ref{Theorem: Main}, we prove
two new results that support this conclusion,
which are motivated by spectral geometry and boundary fractal geometry:
\begin{enumerate}[(i)]
\item {\bf Spectral Geometry.} In {\bf Corollary \ref{Corollary: Spectral Geometry}}, we show that if
$\partial D$ is sufficiently regular, then $D$'s area and $\partial D$'s length
can be recovered almost surely from a single observation of $H_\ka$'s eigenvalues.
This improves Mouzard's Weyl law \cite{Mouzard}
in the case of bounded domains, which showed that $D$'s area can be recovered from $\la_n(H_\ka)$'s
first order asymptotics.
\item {\bf Boundary Fractal Geometry.} In {\bf Corollary \ref{Corollary: Fractal Dimension}}, we show that if
$D$ is simply connected and
$\partial D$ has a well-defined Minkowski dimension $\msf d_M(\partial D)\in[1,2)$, then we can recover
$\msf d_M(\partial D)$ almost-surely from $\msf M_\ka(t)$'s small $t$ asymptotics. This extends a similar result
for $\msf M_0(t)$ due to van den Berg \cite{VanDenBerg}.
\end{enumerate}
See Section \ref{Section: Motivation, Applications} for more details on the context and significance of these results.

In addition to these two applications, Theorem \ref{Theorem: Main} identifies exactly when one "begins to feel"
the presence of $\ka\xi$ in $H_\ka$'s large eigenvalue/eigenfunction asymptotics.
We present the following application of this identification:

\begin{enumerate}[(i)]
\setcounter{enumi}{2}
\item {\bf Variance Recovery.} The parameter $\ka^2$ can be recovered almost surely from a single observation of
$H_\ka$'s eigenvalues; see {\bf Corollary \ref{Corollary: Variance}} for details.
\end{enumerate}

This is rather surprising in light of the fact that, to the best of our knowledge,
$\ka^2$ cannot be recovered almost surely if $\xi$ is replaced by a more regular noise; see Section
\ref{Section: Open Problem Transition} for more details on this point. In fact, the
presence of logarithmic terms in $\msf T_\ka(t)$'s and $\msf M_\ka(t)$'s asymptotics is quite
unusual, since adding a more regular potential to the Laplacian typically induces power-law
corrections to the heat trace; e.g., \cite{SmithH}. Our method of proof shows that the
logarithmic terms emerge from the singularity of the noise $\xi$, and corresponds to the fact that the renormalization in
\eqref{Equation: Smooth AH Limits} is logarithmic in $\eps$;
see Remarks \ref{Remark: Renormalization 1}, \ref{Remark: Renormalization 2}, and \ref{Remark: Renormalization 3} for details.

\begin{remark}
In this context, it is interesting to note
that our results actually imply the presence of an infinite number of logarithmic terms
in $\msf T_\ka(t)$'s and $\msf M_\ka(t)$'s asymptotics. Indeed, the expectation asymptotics
in \eqref{Equation: Main Asymptotic 1} and \eqref{Equation: Main Asymptotic 2}
follow from the more general statements that, as $t\to0$, one has
\[\mbf E\big[\msf T_\ka(t)\big]=\mr e^{\ka^2t\log t/2\pi}\big(\msf T_0(t)+o(1)\big)
\qquad\text{and}\qquad
\mbf E\big[\msf M_\ka(t)\big]=\mr e^{\ka^2(t\log t-t)/2\pi}\big(\msf M_0(t)+o(t)\big).\]
See Section \ref{Section: Intersection Local Times Asymptotics} for more details.
\end{remark}

\subsection{Outline of Proof}
\label{Section: Outline of Proof}

Previous results regarding Question \ref{Question: Effect of Noise}
typically rely on the sophisticated analytic machinery used to construct
the AH's domain and the PAM's solution.
For instance, \cite{Mouzard} studies $\la_n(H_\ka)$'s large $n$ asymptotics using
variational methods, and \cite{BailleulDangMouzard} studies the same quantities
using $\mr e^{-tH_\ka}$'s analytic properties.

In contrast to this, the approach used in this paper is entirely probabilistic.
In fact, at a purely technical level,
our results (including the presence of logarithmic terms) can be viewed as a direct manifestation of the small-time
asymptotics of intersection local times of planar Brownian motions and bridges.
In order to explain this, we now go over a brief outline of the main steps in our proof of Theorem \ref{Theorem: Main}.

\subsubsection{Step 1: Construction of $H_\ka$ and $u_\ka$}

By referring to the constructions carried out in \cite{ChoukVanZuijlen,Labbe,MatsudaVanZuijlen,Mouzard},
in Assumption \ref{Assumption: AH} we take the existence of $H_\ka$ and its
renormalized limit in \eqref{Equation: Smooth AH Limits} for granted, with one specific choice
for the renormalization constant $\msf c_{\ka,\eps}$---see \eqref{Equation: Renormalization}.
We then construct $u_\ka(t,\cdot)$ using $H_\ka$'s semigroup.
We refer to Section \ref{Section: Construction of H and u}
for detailed references.

\begin{remark}
This step is the only potential input from regularity structures, paracontrolled calculus, or any alternative
analytical machinery in this paper.
\end{remark}

\subsubsection{Step 2: Feynman-Kac Formulas}

By combining Step 1 with some technical uniform integrability estimates, we obtain the following smooth approximations
for the expectations and covariances that appear in Theorem \ref{Theorem: Main}:

\begin{proposition}
\label{Proposition: Moments Prelimit}
Recall \eqref{Equation: Smooth AH} and \eqref{Equation: Smooth AH Limits}.
For every $\ka,\eps,t>0$, denote
\begin{align}
\label{Equation: Intro Regularized Moments}
\msf T_{\ka,\eps}(t)=\mr{Tr}\big[\mr e^{-tH_{\ka,\eps}}\big]
\qquad\text{and}\qquad
\msf M_{\ka,\eps}(t)=\int_Du_{\ka,\eps}(t,x)\d x.
\end{align}
For every $\ka>0$,
there exists a constant $\th_\ka>0$ such that
for every $t\in(0,\th_\ka)$,
\begin{align*}
\mbf E\big[\msf T_{\ka}(t)\big]=\lim_{\eps\to0}\mbf E\big[\msf T_{\ka,\eps}(t)\mr e^{-t\msf c_{\ka,\eps}}\big]
\quad\text{and}\quad
\displaystyle \mbf{Var}\big[\msf T_{\ka}(t)\big]=\lim_{\eps\to0}\mbf{Var}\big[\msf T_{\ka,\eps}(t)\mr e^{-t\msf c_{\ka,\eps}}\big];\\
\mbf E\big[\msf M_{\ka}(t)\big]=\lim_{\eps\to0}\mbf E\big[\msf M_{\ka,\eps}(t)\mr e^{-t\msf c_{\ka,\eps}}\big]
\quad\text{and}\quad
\displaystyle \mbf{Var}\big[\msf M_{\ka}(t)\big]=\lim_{\eps\to0}\mbf{Var}\big[\msf M_{\ka,\eps}(t)\mr e^{-t\msf c_{\ka,\eps}}\big].
\end{align*}
\end{proposition}

Proposition \ref{Proposition: Moments Prelimit} is proved in Section \ref{Section: Moments Prelimit}.
The significance of this result is as follows: Since $\xi_\eps$ is smooth, we can use the classical Feynman-Kac formula
(i.e., \eqref{Equation: F-K})
to write $\msf T_{\ka,\eps}(t)$ and $\msf M_{\ka,\eps}(t)$ in terms of simple expectations involving Brownian bridges
and motions in $\mbb R^2$. Thus, Proposition \ref{Proposition: Moments Prelimit} suggests that if we can somehow
make sense of the $\eps\to0$ limits of the renormalized moments of these Feynman-Kac formulas, then this will provide Feynman-Kac
formulas for the moments of $\msf T_\ka(t)$ and $\msf M_\ka(t)$. This seems especially plausible in light of
the results in \cite{GuXu}, which used this exact strategy to provide Feynman-Kac formulas for the moments
of the PAM on $\mbb R^2$.

In this context, the main result in the second step of our proof of Theorem \ref{Theorem: Main}
is to carry out this program. More specifically,
in Proposition \ref{Proposition: Feynman-Kac for T and M} we provide Feynman-Kac formulas
for the expectations and variances of $\msf T_{\ka,\eps}(t)$ and $\msf M_{\ka,\eps}(t)$ in terms of
intersection local times (see Section \ref{Section: SILT and MILT} for definitions and references),
which quantify how often Brownian motion and bridge
paths intersect themselves or each other. The logarithmic contributions
to \eqref{Equation: Main Asymptotic 1} and \eqref{Equation: Main Asymptotic 2}
emerge in this step, from the expected values of approximate self-intersection local times
(see Lemma \ref{Lemma: Smoothed Expectation of SILT}).

\subsubsection{Step 3: Intersection Local Times Asymptotics}

Once Step 2 is completed,
the proof of Theorem \ref{Theorem: Main} reduces to analyzing
the small $t$ asymptotics of the intersection local times that appear in the
Feynman-Kac formulas for $\msf T_{\ka}(t)$'s and $\msf M_{\ka}(t)$'s
moments. This analysis is carried out in Section \ref{Section: Intersection Local Times Asymptotics}.

\subsection{Organization}

The remainder of this paper is organized as follows:
\begin{enumerate}[$\bullet$]
\item In Section \ref{Section: Past, Motivation, Applications},
we survey past results related to Question \ref{Question: Effect of Noise},
and we state our applications of Theorem \ref{Theorem: Main}
 (i.e., Corollaries \ref{Corollary: Spectral Geometry}, \ref{Corollary: Fractal Dimension}, and \ref{Corollary: Variance}).
We take this opportunity to contextualize and motivate these applications.
\item In Section \ref{Section: Open Problems}, we provide a sample of the many interesting
open problems that emerge from our main result and applications.
\item In Sections \ref{Section: Construction of H and u}, \ref{Section: Step 2}, 
and \ref{Section: Intersection Local Times Asymptotics}, we respectively carry out Steps 1, 2, and 3
in the outline of proof in Section \ref{Section: Outline of Proof}.
\item In Section \ref{Section: SILT and MILT}, we first provide the precise definitions of
SILTs and MILTs. Then, we state (and in few cases, prove) technical results regarding the latter that are needed
in Sections \ref{Section: Step 2}
and \ref{Section: Intersection Local Times Asymptotics}.
\item Finally, in Section \ref{Section: Applications}, we prove our two applications of
Theorem \ref{Theorem: Main}, i.e., Corollaries \ref{Corollary: Spectral Geometry}, \ref{Corollary: Fractal Dimension}, and
\ref{Corollary: Variance}.
\end{enumerate}

\subsection*{Acknowledgements}

We thank William Wylie, Ralph Xu, Ruojing Jiang, and Jia Shi for multiple
helpful conversations and references on differential geometry.

We gratefully acknowledge Kasso Okoudjou for pointing
out that our main result can be used to recover the Minkowski dimension of a
fractal boundary (i.e., Corollary \ref{Corollary: Fractal Dimension}), and providing the
reference \cite{VanDenBerg}.

We would like to thank the anonymous referees for their very careful reading of the manuscript---in
particular the errors and typos they pointed out---and
their recommendations to improve the presentation.

\section{Past Results, Motivation, and Applications}
\label{Section: Past, Motivation, Applications}

\subsection{Past Results}

\subsubsection{Localization/Intermittency in Small Eigenvalues/Large Time}
\label{Section: Past 1}

The majority of prior works on Question \ref{Question: Effect of Noise} concern
the effect of $\ka$ on the small eigenvalues and eigenfunctions
of the AH, or the large-time asymptotics of the PAM. These two
problems are closely related thanks to the spectral expansion
\begin{align}
\label{Equation: Spectral Expansion}
u_\ka(t,\cdot)=\sum_{n=1}^\infty\mr e^{-t\la_n(H_\ka)}\big\langle\psi_n(H_\ka),\mbf 1_D\big\rangle\psi_n(H_\ka),\qquad t>0.
\end{align}
Indeed, the leading order contributions
to \eqref{Equation: Spectral Expansion} as $t\to\infty$ are obviously
determined by $\mr e^{-t\la_n(H_\ka)}$'s magnitude and $\psi_n(H_\ka)$'s shape for small $n$.

The main motivation for studying the effect of $\ka$ on
$H_\ka$'s smallest eigenvalues comes from the localization phenomenon,
first described in lattice models by Anderson
\cite{Anderson} (see, e.g., \cite{Stolz} for a survey).
That is, for small $n$, it is expected that the eigenfunction mass
$|\psi_n(H_\ka)|^2$ transitions from being delocalized when $\ka=0$
to being localized when $\ka>0$,
with the latter effect being increasingly pronounced as $\ka$ gets larger.
We refer to \cite[Figures 1 and 2]{AllezChouk} for an illustration in the case where $D$
is a square.
This phenomenon is very well understood in the one-dimensional AH (on $\mbb R$ or an interval),
e.g., \cite{DumazLabbeLoc1,DumazLabbeLoc2}.
Related results for the two-dimensional AH include asymptotics of $H_\ka$'s lowest eigenvalues
\cite{AllezChouk,ChoukVanZuijlen,HsuLabbe,Labbe}, and computations of $H_\ka$'s
density of states \cite{MatsudaVanZuijlen,Matsuda}.

For the PAM, one is interested in the large-time asymptotics of $u_\ka(t,\cdot)$
due to the intermittency phenomenon, first described in lattice models by
G\"artner and Molchanov \cite{GartnerMolchanov,GartnerMolchanov2}
(see, e.g., \cite{KonigBook} for a survey).
That is, as $t\to\infty$, all heat rapidly escapes $D$ through the boundary, whereby $u_0(t,\cdot)\to0$
at rate $\mr e^{-t\la_1(H_0)}$.
In contrast, the PAM's solution $u_\ka(t,\cdot)$ for $\ka>0$ develops
sharp and narrow peak-like formations (which correspond to $\psi_n(H_\ka)'s$
localized shapes)
that grow exponentially fast.
Prior works in this direction concerning the PAM on $\mbb R^2$
or square domains
include studies of $u_\ka(t,\cdot)$'s almost-sure asymptotics
\cite{KonigPerkowskiVanZuijlen} and of $u_\ka(t,\cdot)$'s moments \cite{BassChen,ChenDeyaOuyangTindel,GLGLTAMS}.

\subsubsection{Delocalization in Large Eigenvalues/Small Time}
\label{Section: Past 2}

Alternatively, one might ask how the large eigenvalues and eigenfunctions of
$H_\ka$ are affected by the transition from $\ka=0$ to $\ka>0$. In view of
\eqref{Equation: Spectral Expansion}, this is equivalent to studying
$u_\ka(t,\cdot)$ as $t\to0$. In sharp contrast to localization and intermittency,
the phenomenon of interest in this setting is delocalization. That is, one expects that despite the
presence of the random noise, the behavior of $\la_n(H_\ka),\psi_n(H_\ka)$ for large $n$
and $u_\ka(t,x)$ for small $t$ is not significantly altered by the transition from $\ka=0$ to $\ka>0$.
See \cite{DumazLabbeLoc1,DumazLabbeDeloc1,DumazLabbeDeloc2} for a comprehensive analysis of
the transition from localization to delocalization in the one-dimensional AH,
wherein $H_\ka$'s spectrum becomes increasingly similar to that of $H_0$ at large energies
(i.e., \cite[Theorem 1.7 and Remark 1]{DumazLabbeDeloc1}).
Moving back to the setting of this paper, to the best of our knowledge, there are only three previous works of this nature
that apply to the AH on two- or three-dimensional spaces:

In \cite{Mouzard}, Mouzard consructed $H_\ka$ on a certain class
of two-dimensional compact Riemannian manifolds $M$, which includes bounded domains with smooth
boundary as a special case. Using this construction, he then proved that for any $\ka\geq0$,
\begin{align}
\label{Equation: Mouzard's Weyl Law}
\#\big(\{n\geq1:\la_n(H_\ka)\leq\la\}\big)=\tfrac{\msf{A}(M)}{2\pi}\la+o(\la)\qquad\text{as }\la\to\infty
\end{align}
almost surely, where $\#(S)$ denotes the cardinality of a set $S$,
and $\msf{A}(M)$ denotes the Riemann volume of the manifold $M$.
When $\ka=0$, this reduces to the classical Weyl law
for the Laplacian eigenvalues \cite{Weyl}. In particular, \eqref{Equation: Mouzard's Weyl Law} showed that
the leading order term in $\la_n(H_\ka)$'s asymptotics remains the same
whether $\ka=0$ or $\ka>0$.

\begin{remark}
Using standard Abelian/Tauberian theorems
(e.g., \cite[Theorem 2, Page 445]{Feller}), \eqref{Equation: Mouzard's Weyl Law} is equivalent
to the trace asymptotic
\begin{align}
\label{Equation: Mouzard's Weyl Law Abelian}
\mr{Tr}\big[\mr e^{-tH_\ka}\big]=\tfrac{\msf{A}(M)}{2\pi}t^{-1}+o(t^{-1})\qquad\text{as }t\to0
\end{align}
for all $\ka\geq0$. If we restrict to the special case where $D$ is a bounded domain,
then this implies \eqref{Equation: Mouzard Law Intro}.
\end{remark}

Next, in \cite{BailleulDangMouzard}, Bailleul, Dang, and Mouzard provided
several extensions and refinements of the results in \cite{Mouzard} in the special case of boundaryless manifolds.
Among many other things,
they proved the following:
Consider $H_\ka$ on some closed Riemannian compact
surface $M$. For every $\eps>0$,
\begin{align}
\label{Equation: Tauberian Weyl Law}
\mr{Tr}\big[\mr e^{-tH_\ka}\big]=\tfrac{\msf{A}(M)}{2\pi}t^{-1}+O(t^{-1/2-\eps})\qquad\text{as }t\to0.
\end{align}
almost surely---see the first display in \cite[Section 4.4]{BailleulDangMouzard}.

Finally, in \cite{MatsudaVanZuijlen}, Matsuda and van Zuijlen constructed
$H_\ka$ with general singular noise on bounded domains $D\subset \mbb R^d$ ($d\geq1$) such that $\partial D$ is Lipschitz,
with either Dirichlet or Neumann boundary condition.
Then, on the way to calculating $H_\ka$'s density of states, they extended Mouzard's Weyl law
(i.e., \eqref{Equation: Mouzard's Weyl Law} and \eqref{Equation: Mouzard's Weyl Law Abelian})
for all of these models;
see \cite[Proposition 5.17]{MatsudaVanZuijlen}.

\begin{remark}
\label{Remark: Generality}
In light of the generality of these results, it is natural to ask
if Theorem \ref{Theorem: Main} and its corollaries (the latter of which are stated below
in Section \ref{Section: Motivation, Applications}) are true in the same generality.
The limitations on the setting considered in this paper comes from our method
of proof. We expect that, up to overcoming some potentially challenging technical difficulties
(e.g., Open Problem \ref{Open Problem: General SILT and MILT}),
the proof method in this paper could be extended to some, but not all of the
settings studied in \cite{BailleulDangMouzard,Matsuda,MatsudaVanZuijlen}.
See Section \ref{Section: Open Problems - General} for more details on this point.
\end{remark}

\subsection{Motivation and Applications}
\label{Section: Motivation, Applications}

The Weyl law stated in \eqref{Equation: Tauberian Weyl Law}
naturally leads to a tantalizing open-ended question,
which was the main motivation for our work:

\begin{question}
\label{Question: Motivation}
Are there more features of the Dirichlet AH/PAM on $D$ (i.e., beyond the Weyl law)
that remain unchanged when going from $\ka=0$ to $\ka>0$? More generally, when
do we begin to "feel" the presence of $\ka>0$ in the large-$n$ asymptotics of
the eigenvalues $\la_n(H_\ka)$?
\end{question}

In this context, the applications of our result are inspired by some of the most important classical
developments regarding the $t\to0$ asymptotics of $\msf T_0(t)$ and $\msf M_0(t)$:

\subsubsection{Spectral Geometry}

Our first application is motivated by the
classical fact that if $\partial D$ is sufficiently regular, then an impressive amount of information about $D$'s geometry can be
recovered from $H_0$'s eigenvalues alone.
For instance, it has been known since the 1950's (and conjectured since the 1910's)
that $H_0$'s eigenvalues also determine $D$'s boundary length:

\begin{theorem}[\cite{Kac,Pleijel}]
\label{Theorem: Smooth Heat Trace}
If $\partial D$ is sufficiently regular (e.g., smooth), then
\begin{align}
\label{Equation: Smooth Heat Trace}
\msf T_0(t)=\tfrac{\msf{A}(D)}{2\pi}t^{-1}
-\tfrac{\msf{L}(\partial D)}{4\sqrt{2\pi}}t^{-1/2}+O(1)\qquad\text{as }t\to0,
\end{align}
where $\msf L(\partial D)$ is the length of the curve
$\partial D$.
\end{theorem}

\begin{remark}
We note that \eqref{Equation: Smooth Heat Trace} is not the full extent of the information about $D$ that can be
recovered from $\msf T_0(t)$; see Section \ref{Section: Open Problem - Complete Expansions} for more details.
\end{remark}

It is thus natural to wonder if one could improve \eqref{Equation: Mouzard's Weyl Law}/\eqref{Equation: Mouzard's Weyl Law Abelian} to such
an extent that $\msf{L}(\partial D)$ or more could be recovered from $\la_n(H_\ka)$'s eigenvalues
when $\ka>0$.
Our first application of Theorem \ref{Theorem: Main} is to show that this is indeed the case:

\begin{corollary}
\label{Corollary: Spectral Geometry}
Suppose that the hypotheses in Theorem \ref{Theorem: Main}
hold, and that $\msf T_0(t)$ has the asymptotic expansion in \eqref{Equation: Smooth Heat Trace}.
Let $t_1>t_2>\cdots>0$ be a vanishing sequence of positive numbers.
The following holds for any fixed $\ka>0$:
\begin{enumerate}[(1)]
\item If $t_n\leq cn^{-1/2-\eps}$ for some fixed $c,\eps>0$, then
\[\lim_{n\to\infty}2\pi t_n\,\msf T_\ka(t_n)=\msf{A}(D)\qquad\text{almost surely}.\]
\item If $t_n\leq cn^{-1-\eps}$ for some fixed $c,\eps>0$, then
\[\lim_{n\to\infty}4\sqrt{2\pi t_n}\left(\tfrac{\msf{A}(D)}{2\pi}t_n^{-1}-\msf T_\ka(t_n)\right)=\msf{L}(\partial D)\qquad\text{almost surely}.\]
\end{enumerate}
\end{corollary}

See Section \ref{Section: Proof of Spectral Geometry} for a proof.

\subsubsection{Boundary Fractal Geometry}

Our second application is inspired by the following seminal result
due to van den Berg:

\begin{theorem}[\cite{VanDenBerg}]
\label{Theorem: Fractal Dimension Heat Content}
Suppose that $D$ is simply connected and that $\partial D$ is Minkowski nondegenerate with
Minkowski dimension $\msf d_M(\partial D)\in[1,2)$. That is,
\[0<\liminf_{r\to0}\frac{\msf{A}(\partial D_r)}{r^{2-\msf d_M(\partial D)}}\leq
\limsup_{r\to0}\frac{\msf{A}(\partial D_r)}{r^{2-\msf d_M(\partial D)}}<\infty,\]
where we define
\[\partial D_r=\left\{x\in\mbb R^2:\inf_{y\in\partial D}|x-y|<r\right\}.\]
There exists constants $c\in(1,\infty)$ and $\th>0$ such that
\begin{align}
\label{Equation: Fractal Dimension Heat Content}
c^{-1}t^{1-\msf d_M(\partial D)/2}\leq\msf{A}(D)-\msf M_0(t)\leq ct^{1-\msf d_M(\partial D)/2},\qquad t\in(0,\th).
\end{align}
\end{theorem}

Thus, the small $t$ asymptotics of $\msf M_0(t)$ uniquely identify $\msf d_M(\partial D)$.
Our main result implies that this can also be done with $\msf M_\ka(t)$:
\begin{corollary}
\label{Corollary: Fractal Dimension}
Suppose that the hypotheses in Theorem \ref{Theorem: Main}
hold, and that $\msf M_0(t)$ satisfies \eqref{Equation: Fractal Dimension Heat Content}.
Let $t_1>t_2>\cdots>0$ be a vanishing sequence of positive numbers.
If $t_n\leq cn^{-1/\msf d_M(\partial D)-\eps}$ for some fixed $c,\eps>0$,
then for every $\ka>0$,
\[\lim_{n\to\infty}2-2\frac{\log\big(\msf{A}(D)-\msf M_\ka(t_n)\big)}{\log t_n}=\msf d_M(\partial D)\qquad\text{almost surely}.\]
\end{corollary}

See Section \ref{Section: Proof of Fractal Dimension} for a proof.

\begin{remark}
There exists asymptotics of $\msf{A}(D)-\msf M_0(t)$ that are more precise
than \eqref{Equation: Fractal Dimension Heat Content} when $\partial D$ is
a specific fractal or regular;
see \cite[Section 1.2]{Rozanova-PierratTeplyaevWinterZahle} and references therein for
a survey of results of this nature.
For instance, if $\partial D$ is smooth, then
\[\msf M_0(t)=\msf{A}(D)-\tfrac{\sqrt 2\msf{L}(\partial D)}{\sqrt \pi}t^{1/2}+\tfrac{\pi\chi(D)}{2}t+O(t^{3/2})\qquad\text{as }t\to0\]
where $\chi(D)$ denotes the Euler characteristic;
see \cite{VanDenBergLeGall}.
In this context,
our main result also implies that one can recover $\msf{A}(D)$ and $\msf{L}(\partial D)$
almost surely from the small $t$ asymptotics of $\msf M_\ka(t)$.
Given that this is very similar to the result involving $\msf T_\ka(t)$ in Corollary \ref{Corollary: Spectral Geometry},
we omit this statement in the interest of brevity.
\end{remark}

\subsubsection{Variance Recovery}

Lastly, we confirm with the following result that $\ka>0$ can be recovered from $H_\ka$'s eigenvalues:

\begin{corollary}
\label{Corollary: Variance}
Suppose that the hypotheses in Theorem \ref{Theorem: Main}
hold.
Let $t_1>t_2>\cdots>0$ be a vanishing sequence of positive numbers
such that $t_n\leq\tilde c\mr e^{-cn^{1/2+\eps}}$ for some fixed $\tilde c,c,\eps>0$.
For every $\ka>0$,
\[\lim_{n\to\infty}\tfrac{4\pi^2}{\msf{A}(D)\log t_n}\Big(\msf T_\ka(t_n)-\msf T_0(t_n)\Big)=\ka^2\qquad\text{almost surely}.\]
\end{corollary}

See Section \ref{Section: Proof of Spectral Geometry} for a proof.

\begin{remark}
Theorem \ref{Theorem: Main} implies that
$\ka$ can also be recovered from the $t\to0$ asymptotics of $\frac{2\pi}{\msf A(D)t\log t}\big(\msf M_\ka(t)-\msf M_0(t)\big)$;
we omit this statement for brevity.
\end{remark}

\section{Open Problems}
\label{Section: Open Problems}

Together with the pioneering works \cite{BailleulDangMouzard,MatsudaVanZuijlen,Mouzard}
outlined in Section \ref{Section: Past 2}, the
main result, applications, and method of proof in this paper naturally
lead to a large number of interesting open problems, which we hope will
lay the foundation for an active research program. In this section, we share
a sample of the problems that we find the most promising.

\subsection{Complete Expansion for $\ka>0$}
\label{Section: Open Problem - Complete Expansions}

It is known that if $\partial D$ is sufficiently regular,
then there exists an infinite sequence of constants $c_0,c_1,c_2,\ldots\in\mbb R$ such that
\begin{align}
\label{Equation: Complete Expansion for T0}
\msf T_0(t)=\sum_{n=0}^\infty c_n\,t^{-1+n/2}\qquad\text{as }t\to0;
\end{align}
see, e.g., \cite[Section 3.0]{Gilkey}. The constants $c_n$---the first
two of which were stated in Theorem
\ref{Theorem: Smooth Heat Trace}---are determined by
various features of $D$'s geometry.
For instance, $c_2=\frac{\chi(D)}6$, where $\chi(D)$ denotes $D$'s
Euler characteristic;
see, e.g., \cite[Page 468]{Smith} and references
therein for exact formulas for $c_3,\ldots,c_6$. 
In this context, our first open problem consists of understanding how the full expansion
\eqref{Equation: Complete Expansion for T0} is affected by the transition
from $\ka=0$ to $\ka>0$:

\begin{openproblem}
\label{Open Problem: Complete Expansions}
Let $\ka>0$.
Does there exist an asymptotic expansion
\begin{align}
\label{Equation: Complete Expansion for ka>0}
\msf T_\ka(t)-\msf T_0(t)=\sum_{n=0}^{\infty}\msf a_n(\ka,t)\qquad\text{as }t\to0,
\end{align}
where $\msf a_n(\ka,t)$ are possibly random coefficients of decreasing size
(e.g., for every $n\geq0$, one has $\msf a_{n+1}(\ka,t)=o\big(\msf a_{n}(\ka,t)\big)$ as $t\to0$)?
Can one characterize the joint distributions (over $n$ and $t$) of the processes $\msf a_n(\ka,t)$?
\end{openproblem}

Our main result only partially address this problem, as follows:
\begin{enumerate}[(1)]
\item $\msf a_0(\ka,t)=\tfrac{\ka^2\msf{A}(D)}{4\pi^2}\log t$ is deterministic.
\item The variance of $\sum_{n\geq1}\msf a_n(\ka,t)$ is of order $O(1)$ as $t\to0$.
\end{enumerate}
At best, the technique used in this paper could provide insights into the joint moments
of the $\msf a_n(\ka,t)$'s. That being said, the pathwise Feynman-Kac formula for the two-dimensional
PAM on a box proved by K\"onig, Perkowski and van Zuijlen in
\cite[Theorem 2.23]{KonigPerkowskiVanZuijlen} seems to suggest that computing
the first few terms in
\eqref{Equation: Complete Expansion for ka>0} exactly might still be possible using probabilistic arguments.

\subsection{What Can One "Hear" About the Shape of the AH's Domain?}

The geometric information contained in the constants $c_n$
in \eqref{Equation: Complete Expansion for T0} naturally leads to
the following question: {\it "Does the spectrum of $H_0$
characterize $D$?"} For several decades, this was considered
a very important open problem in analysis, which was widely
known as {\it "Can one hear the shape of a drum?"}
due to an influential paper by Kac \cite{Kac}. In the early 1990's,
the problem was settled in the negative by Gordon, Webb, and
Wolpert \cite{GordonWebbWolpert}, who showed that
there exists non-isometric planar domains that are isospectral
(i.e., domains on which $H_0$ has the same spectrum). One can then ask:

\begin{openproblem}
\label{Open Problem: Shape of a Drum}
Are there non-isometric planar domains
on which $H_\ka$'s eigenvalues have the same joint distribution when $\ka>0$?
\end{openproblem}

Corollaries \ref{Corollary: Spectral Geometry} and \ref{Corollary: Variance} imply that if $\partial D$
is sufficiently regular, then $H_\ka$'s eigenvalue point processes
(for different $D$'s and $\ka$'s) become mutually singular as soon as
one changes $D$'s area, $\partial D$'s length, or the strength of the noise $\ka$. This constrains the possible
constructions that would answer Open Problem \ref{Open Problem: Shape of a Drum}.

However, in addition to leading to exact analogues of classical spectral geometry problems,
injecting randomness in $H_0$ could potentially lead to entirely new phenomena.
For instance, while two domains $D$ are either isospectral or not
in the deterministic problem, the spectra of $H_\ka$ for $\ka>0$
over different domains could conceivably interpolate between these two extremes:

\begin{openproblem}
\label{Open Problem: Shape of a Drum 2}
Are there pairs of non-isometric domains $D_1$ and $D_2$ on which
$H_\ka$'s eigenvalue point processes are neither mutually singular nor equal in distribution?
If so, what are the possible relationships between the two laws (e.g., mutual absolute continuity, or something else)?
\end{openproblem}

Corollaries \ref{Corollary: Spectral Geometry} and \ref{Corollary: Variance} once again constrain the possible relationships
that such hypothetical domains might have.

\subsection{General Manifolds, Noises, and Boundary Conditions}
\label{Section: Open Problems - General}

Following-up on Remark \ref{Remark: Generality}, it is natural to ask the following:

\begin{openproblem}
\label{Open Problem: General}
Do suitable analogues of Theorem \ref{Theorem: Main} and Corollaries \ref{Corollary: Spectral Geometry},
\ref{Corollary: Fractal Dimension}, and \ref{Corollary: Variance} still hold if one or several of the following changes are made?
\begin{enumerate}[(1)]
\item $D$ is replaced by a compact Riemannian manifold of any dimension.
\item $\xi$ is replaced by a more general singular noise
(e.g., $\xi$ could be Gaussian with covariance $\mbf E[\xi(x)\xi(y)]=\mf K(x,y)$
for the Riesz kernel $\mf K(x,y)=|x-y|^{-\om}$, $\om>0$, or the fractional
kernel $\mf K(x,y)=\prod_{i}|x_i-y_i|^{-\om_i}$, $\om_i>0$).
\item Assuming there is a boundary, the Dirichlet boundary condition is replaced by
another boundary condition (e.g., Neumann, Robin, or mixed).
\end{enumerate}
\end{openproblem}

Among other things, a positive resolution of this problem would
improve every known Weyl law for the AHs (i.e., the results of \cite{BailleulDangMouzard,MatsudaVanZuijlen,Matsuda}
outlined in Section \ref{Section: Past 2}). 
With this out of the way, two important remarks are in order:

\begin{remark}
The currently-known techniques used to construct the AH and PAM
only work for a certain class of singular operators and PDEs, which are known as
subcritical (e.g., \cite[Page 272]{Hairer}). Among other things, this puts constraints
on the possible combinations of noises and manifolds for which $H_\ka$ and $u_\ka$ can be
defined. (For instance, the AH and PAM with white noise are not expected to
make sense on manifolds whose dimension is greater or equal to $4$.)
Modulo these kinds of limitations,
we expect that the short answer to Open Problem \ref{Open Problem: General} is essentially positive.
\end{remark}

\begin{remark}
\label{Remark: When is Our Method Applicable}
That being said, we do not expect that the method of proof
outlined in Section \ref{Section: Outline of Proof} can be used to settle Open Problem \ref{Open Problem: General} in every
case for which the AH and PAM can be defined. This is because, if the noise is too irregular,
then the moments of $\mr{Tr}\big[\mr e^{-tH_\ka}\big]$ and $u_\ka(t,x)$ are infinite for all $t>0$,
thus rendering a statement analogous to Theorem \ref{Theorem: Main} impossible.
(For instance, it is known that the moments of the PAM with white noise are always infinite in three dimensions;
e.g., \cite[Paragraph following Theorem 2]{Labbe}
and \cite[Remark 1.7]{GuXu}.) That said, we expect that our method could be used whenever
the moments of $\mr{Tr}\big[\mr e^{-tH_\ka}\big]$ and $u_\ka(t,x)$ are finite for all $t>0$,
or at least for all
$t\in(0,\th_\ka)$ smaller than some threshold $\th_\ka>0$.
\end{remark}

In this context, we point out that
the main obstacle to using the proof method
outlined in Section \ref{Section: Outline of Proof}
with more general manifolds is as follows: To the best of our knowledge,
intersection local times have so far only been constructed for Brownian motions/bridges
in $\mbb R^n$. In particular,
following-up on Remark \ref{Remark: Generality}, this explains why we have restricted our
attention to the AH and PAM with white noise and Dirichlet boundary condition on flat domains in this paper. That being said,
the extent of the improvements on the Weyl law in \eqref{Equation: Mouzard Law Intro}
achieved in this paper motivates studying intersection
local times in curved spaces. For instance:

\begin{openproblem}
\label{Open Problem: General SILT and MILT}
Let $M$ be a compact Riemannian surface, with or without boundary.
Construct, and then study the small $t$ asymptotics of the intersection local times
of Brownian motions and bridges (with reflection if there is a boundary) on $M$. That is, replicate all the results in Section
\ref{Section: SILT and MILT} on general surfaces.
\end{openproblem}

Once Open Problem \ref{Open Problem: General SILT and MILT} is solved,
adapting our proof of Theorem \ref{Theorem: Main} to this more general setting
should be relatively straightforward.

\subsection{Transition of Asymptotics as $\eps\to0$}
\label{Section: Open Problem Transition}

One of the most striking features of Corollary \ref{Corollary: Variance}
is that, at first glance, the eigenvalues of $H_\ka$ appear to contain more information
than the eigenvalues of $H_{\ka,\eps}$ for any fixed $\eps>0$.
More specifically, suppose that $\msf T_0(t)$
satisfies \eqref{Equation: Complete Expansion for T0}, and let $\eps>0$ be fixed.
Recall the definition of $\msf T_{\ka,\eps}(t)$ in \eqref{Equation: Intro Regularized Moments}.
As $\xi_\eps$ is smooth, a straightforward calculation using the classical Feynman-Kac formula
(e.g., \eqref{Equation: F-K}) yields
\begin{multline}
\label{Equation: Smooth Potential Trace Asymptotics}
\msf T_{\ka,\eps}(t)=\msf T_0(t)-\frac{\ka}{2\pi}\int_D\xi_\eps(x)\d x+o(1)\\
=\tfrac{\msf{A}(D)}{2\pi}t^{-1}
-\tfrac{\msf{L}(\partial D)}{4\sqrt{2\pi}}t^{-1/2}+\left(\tfrac{\chi(D)}6-\frac{\ka}{2\pi}\int_D\xi_\eps(x)\d x\right)+o(1)\qquad\text{as }t\to0.
\end{multline}
This implies that $\msf A(D)$ and $\msf L(\partial D)$ can
be recovered from $H_{\ka,\eps}$'s eigenvalues in a way that is analogous to Corollary \ref{Corollary: Spectral Geometry}-(1) and -(2):
\[\lim_{t\to0}2\pi t\,\msf T_{\ka,\eps}(t)=\msf{A}(D)
\qquad\text{and}\qquad
\lim_{t\to0}4\sqrt{2\pi t}\left(\tfrac{\msf{A}(D)}{2\pi}t^{-1}-\msf T_{\ka,\eps}(t)\right)=\msf{L}(\partial D).\]
However, there does not appear to be an analogue of Corollary \ref{Corollary: Variance} for $\eps>0$.
Indeed, the dependence on $\ka$ in \eqref{Equation: Smooth Potential Trace Asymptotics}
is contained in the constant order term, which is Gaussian with mean $\tfrac{\chi(D)}6$ and variance
$\frac{\ka^2}{4\pi^2}\iint_{D^2}p_\eps(x-y)\dd x\dd y.$

Given this sharp contrast between the behavior of $\msf T_\ka(t)$ and $\msf T_{\ka,\eps}(t)$
as $t\to0$, it would be interesting to better understand how these asymptotics gradually
change nature as $\eps\to0$. For this purpose, we propose the following:

\begin{openproblem}
Let $\msf e:(0,1)\to(0,1)$ be such that $\msf e(t)\to0$ as $t\to0$.
Calculate the small $t$ asymptotics of $\msf T_{\ka,\msf e(t)}(t)$ as $t\to0$.
\end{openproblem}

We expect that there exists some critical rate of decay $\mf d(t)$ such that
\begin{enumerate}[(1)]
\item if $\msf e(t)\gg\mf d(t)$, then $\msf T_{\ka,\msf e(t)}(t)$'s asymptotics are similar to
that of $\msf T_{\ka,\eps}(t)$ for fixed $\eps$;
\item if $\msf e(t)\ll\mf d(t)$, then $\msf T_{\ka,\msf e(t)}(t)$'s asymptotics are similar to
that of $\msf T_\ka(t)$.
\end{enumerate}
Identifying the precise nature of this hypothetical "phase transition" would presumably shed light on the
transition between the Gaussian constant order term in \eqref{Equation: Smooth Potential Trace Asymptotics}
and the emergence of the deterministic term $\tfrac{\ka^2\msf{A}(D)}{4\pi^2}\log t$ in \eqref{Equation: Main Asymptotic 1}.

\section{Theorem \ref{Theorem: Main} Step 1: Construction of $H_\ka$ and $u_\ka$}
\label{Section: Construction of H and u}

As hinted at in \eqref{Equation: Smooth AH Limits},
we assume in this paper that $H_\ka$ is constructed as the renormalized limit of the smooth approximations $H_{\ka,\eps}$.
Given that the Feynman-Kac formula for $H_{\ka,\eps}$ is a fundamental aspect of our approach,
we take this opportunity to define $H_{\ka,\eps}$ via its semigroup:

\begin{definition}
\label{Definition: Brownian}
We use $B$ to denote a standard Brownian motion on $\mbb R^2$. For any $x,y\in\mbb R^2$ and $t>0$,
we let $B^x=(B|B(0)=x)$ denote the Brownian motion started at $x$, and let $B^{x,y}_t=(B|B(0)=x,~B(t)=y)$
denote the Brownian bridge from $x$ to $y$ on the time interval $[0,t]$. Finally, when we require multiple
independent copies of Brownian motion, we use the notation ${_1}B,{_2}B,{_3}B,\ldots$
(in particular, we use ${_i}B^{x_i}$ and ${_i}B^{x_i,y_i}_t$, $i\geq1$, to denote independent Brownian motions
started at $x_i$ and independent Brownian bridges from $x_i$ to $y_i$ on the time interval $[0,t]$).
\end{definition}

\begin{definition}
\label{Definition: Hitting Time}
For any stochastic process $Z$ that takes values in $\mbb R^2$,
we denote the hitting time of $D$'s complement as
\[\tau_D(Z)=\inf\big\{t>0:Z(t)\not\in D\big\},\]
with the convention that $\tau_D(Z)=\infty$ if $Z$ never exits $D$.
\end{definition}

\begin{definition}
Let $\xi:L^2(\mbb R^2)\to\mbb R$ be the Gaussian process with mean zero and
covariance \eqref{Equation: Inner Product Covariance}. Viewing $\xi$ as
a random Schwartz distribution, for every $\eps>0$, we define $\xi_\eps:\mbb R^2\to\mbb R$
as
\[\xi_\eps(x)=(\xi*p_{\eps/2})(x)=\xi\big(p_{\eps/2}(x-\cdot)\big),\qquad x\in\mbb R^2,\]
where $*$ denotes the convolution.
\end{definition}

\begin{remark}
\label{Remark: General xi eps Covariance}
For every $\eps>0$,
$\xi_\eps$ is a continuous function since $p_{\eps/2}$ is smooth.
Moreover, $\xi_\eps$ is a Gaussian process with mean zero and covariance
\eqref{Equation: xi epsilon covariance}, as can be shown with the following more general calculation: For any $\eps,\tilde\eps>0$
and $x,y\in\mbb R^2$,
\begin{multline}
\label{Equation: General xi eps Covariance}
\mbf E\big[\xi_\eps(x)\xi_{\tilde\eps}(y)\big]
=\mbf E\big[\xi\big(p_{\eps/2}(x-\cdot)\big)\xi\big(p_{\tilde\eps/2}(y-\cdot)\big)\big]\\
=\int_{\mbb R^2}p_{\eps/2}(x-z)p_{\tilde\eps/2}(y-z)\d z
=p_{(\eps+\tilde\eps)/2}(x-y).
\end{multline}
\end{remark}

\begin{definition}
For any $\ka,\eps,t>0$, we let $\mc K_{\ka,\eps}(t)$ denote the integral operator on $L^2(D)$ with kernel
\begin{align}
\label{Equation: F-K}
\mc K_{\ka,\eps}(t;x,y)=p_t(x-y)\mbf E_B\left[\mbf 1_{\{\tau_D(B^{x,y}_t)>t\}}\exp\left(-\ka\int_0^t\xi_\eps\big(B^{x,y}_t(r)\big)\d r\right)\right],\qquad x,y\in D,
\end{align}
where we assume that $B^{x,y}_t$ and $\xi_\eps$ are independent, and we use $\mbf E_B$ to
denote the (conditional) expectation with respect to the randomness in $B$ only.
\end{definition}

We record the following elementary property of the $\mc K_{\ka,\eps}(t)$'s:

\begin{proposition}
\label{Proposition: Smoothed Strong-Continuity}
For every $\ka,\eps>0$,
the family $\{\mc K_{\ka,\eps}(t):t>0\}$ is a strongly continuous semigroup
of symmetric compact operators on $L^2(D)$.
\end{proposition}
\begin{proof}
By \cite[Chapter 1, (3.14) and (3.15)]{Sznitman}, it suffices to show that the function
$V=\ka\xi_\eps$ is in the Kato class; see \cite[Chapter 1, (2.4)]{Sznitman} for a definition.
Since $\xi_\eps$ is continuous and $D$ is bounded, $\xi_\eps$ is bounded. Thus,
\[\lim_{t\to0}\sup_{x\in D}\mbf E_B\left[\int_0^t\big|\ka\xi_\eps\big(B^x(r)\big)\big|\d r\right]\leq \lim_{t\to0}t\ka\|\xi_\eps\|_{L^\infty(D)}=0,\]
which concludes the proof by definition of the Kato class.
\end{proof}

We can now define $H_{\ka,\eps}$ and $u_{\ka,\eps}$ as follows:

\begin{definition}
For every $\eps,\ka>0$,
we define $H_{\ka,\eps}$ as the infinitesimal generator of $\{\mc K_{\ka,\eps}(t):t>0\}$ on $L^2(D)$
(e.g., \cite[Chapter II, Definition 1.2]{EngelNagel}). Then, we define
\[u_{\ka,\eps}(t,x)=\mr e^{-tH_{\ka,\eps}}\mbf 1_D(x)=\mbf E\left[\mbf 1_{\{\tau_D(B^x)>t\}}\exp\left(-\int_0^t\xi_\eps\big(B^x(r)\big)\d r\right)\right],\qquad t>0,~x\in D.\]
\end{definition}

The following properties of $H_{\ka,\eps}$ are elementary:

\begin{proposition}
For every $\eps,\ka>0$,
the operator $H_{\ka,\eps}$ is self-adjoint on $L^2(D)$.
Its spectrum $\la_1(H_{\ka,\eps})\leq\la_2(H_{\ka,\eps})\leq\cdots$ is purely discrete, bounded below, and has no accumulation point.
Moreover, the corresponding eigenfunctions $\psi_1(H_{\ka,\eps}),\psi_2(H_{\ka,\eps}),\ldots$
form an orthonormal basis of $L^2(D)$. Finally, for any $f\in L^2(D)$,
\begin{align}
\label{Equation: H ka eps Spectral Expansion}
\mr e^{-t H_{\ka,\eps}}f=\sum_{n=1}^\infty\mr e^{-t\la_n(H_{\ka,\eps})}\big\langle\psi_n(H_{\ka,\eps}),f\big\rangle\psi_n(H_{\ka,\eps}).
\end{align}
\end{proposition}
\begin{proof}
See \cite[Chapter 1, Remark 4.14]{Sznitman} for the details on $H_{\ka,\eps}$ being self-adjoint.
See \cite[Chapter 3, (1.47)]{Sznitman} for the claims regarding $H_{\ka,\eps}$'s eigenvalues
and eigenfunctions (as stated, \cite[Chapter 3, (1.47)]{Sznitman} applies to Schr\"odinger operators
with nonnegative potentials, but this can be trivially extended to potentials that are bounded
below---such as $\ka\xi_\eps$---by adding a constant to the potential).
Once the properties of $H_{\ka,\eps}$'s spectrum are established, \eqref{Equation: H ka eps Spectral Expansion}
follows from the facts that the right-hand side of \eqref{Equation: H ka eps Spectral Expansion} is
clearly a semigroup that generates $H_{\ka,\eps}$, and that semigroups are uniquely determined by
their generators (e.g., \cite[Chapter II, Theorem 1.4]{EngelNagel}).
\end{proof}

With this in hand, we may now state our assumption regarding the existence of $H_\ka$ as a renormalized limit
of $H_{\ka,\eps}$, and then define $u_{\ka}$ using its semigroup:

\begin{assumption}
\label{Assumption: AH}
There exists a random self-adjoint operator $H_\ka$ on $L^2(D)$
that satisfies the following conditions:
\begin{enumerate}[(1)]
\item $H_\ka$'s spectrum $\la_1(H_\ka)\leq\la_2(H_\ka)\leq\cdots$ is purely discrete,
bounded below, and without accumulation point. The corresponding eigenfunctions $\psi_1(H_\ka),\psi_2(H_\ka),\ldots$
form an orthonormal basis of $L^2(D)$.
\item Define the renormalization constant
\begin{align}
\label{Equation: Renormalization}
\msf c_{\ka,\eps}=\tfrac{\ka^2}{2\pi}\log(1/\eps),\qquad\eps>0.
\end{align}
For every $n\geq1$,
\begin{equation}
\label{Equation: Assumption Convergence of Eigenvalues}
\lim_{\eps\to0}\big(\la_n(H_{\ka,\eps})+\msf c_{\ka,\eps}\big)=\la_n(H_\ka)\qquad\text{in probability.}
\end{equation}
\item It is possible to choose $\{\psi_n(H_{\ka,\eps}):n\geq1\}$ and $\{\psi_n(H_\ka):n\geq1\}$
in such a way that for every $n\geq1$,
\begin{equation}
\label{Equation: Assumption Convergence of Eigenfunctions}
\lim_{\eps\to0}\|\psi_n(H_{\ka,\eps})-\psi_n(H_\ka)\|^2=0\qquad\text{in probability}.
\end{equation}
\end{enumerate}
Finally, for every $\ka,t>0$, we define $u_\ka(t,\cdot):D\to\mbb R$ as the function in \eqref{Equation: Spectral Expansion}.
\end{assumption}

We end this section with three remarks, the latter two of which
comment on previous works that constructed $H_\ka$
in a way that is consistent with Assumption \ref{Assumption: AH}.

\begin{remark}
\label{Remark: Choosing the Eigenfunctions}
The caveat that we must "choose" the eigenfunctions in such a way that \eqref{Equation: Assumption Convergence of Eigenfunctions} holds is caused by the following subtlety: If we only require that the eigenfunctions of a self-adjoint operator form
an orthonormal basis of $L^2(D)$, then there is some arbitrariness in how exactly we choose the eigenfunctions as basis elements of their respective eigenspace. As pointed out in \cite[Remark 1.1]{Labbe}, we can reformulate \eqref{Equation: Assumption Convergence of Eigenfunctions} without making references to eigenfunction choice
as follows: Suppose that $\psi_n(H_\ka),\ldots,\psi_{n+m}(H_\ka)$ is an orthonormal basis of the eigenspace of some eigenvalue $\la_{n}(H_{\ka})$
(with $m=0$ if $\la_{n}(H_{\ka})$ has multiplicity one). Then, the unit ball of the subspace spanned by $\psi_n(H_{\ka,\eps}),\ldots,\psi_{n+m}(H_{\ka,\eps})$
converges in probability to the unit ball of the subspace spanned by $\psi_n(H_\ka),\ldots,\psi_{n+m}(H_\ka)$ with respect to the Hausdorff metric induced by the
$L^2$-norm on $L^2(D)$.
\end{remark}

\begin{remark}
\label{Remark: Construction of Hk Previous Works}
As shown in \cite[Lemma 4.8]{Labbe}, the convergence
$H_{\ka,\eps}\to H_\ka$ in the sense of \eqref{Equation: Assumption Convergence of Eigenvalues}
and \eqref{Equation: Assumption Convergence of Eigenfunctions}/Remark \ref{Remark: Choosing the Eigenfunctions} can be reduced to
the convergence in norm of the resolvent operators $(H_{\ka,\eps}-z)^{-1}\to(H_\ka-z)^{-1}$
for some $z\in\mbb C\setminus\mbb R$, together with some uniform compactness estimate.
This was proved on the square $D=(0,L)^2$ for some $L>0$ in \cite[Proof of Theorem 1, Page 3225]{Labbe} (see \cite{ChoukVanZuijlen} for a similar construction), on domains with smooth boundary in \cite[Proposition 2.14]{Mouzard},
and on arbitrary domains in \cite[Theorem 5.3]{MatsudaVanZuijlen}.
\end{remark}

\begin{remark}
\label{Remark: Renormalization 1}
One could replace our choice of $\msf c_{\ka,\eps}$ in \eqref{Equation: Renormalization} by
\begin{align}
\label{Equation: Alternate Renormalization}
\msf c_{\ka,\eps}=\tfrac{\ka^2}{2\pi}\log(1/\eps)+c+o(1),\qquad\eps\to0
\end{align}
for any arbitrary constant $c\in\mbb R$ and vanishing correction $o(1)$
without impacting the conclusion of Assumption \ref{Assumption: AH}. This partly explains why 
the renormalization constant for the two-dimensional AH and PAM is often written as
$\frac{\ka^2}{2\pi}\log(1/\eps)+O(1)$ (e.g., \cite[(1.2)]{HairerLabbe} or \cite[Section 4.1]{Labbe}).
If we were to choose \eqref{Equation: Alternate Renormalization} instead, then our main result would remain unchanged.
That being said, not all of the works cited in Remark \ref{Remark: Construction of Hk Previous Works} identify that the renormalization must be of the form \eqref{Equation: Alternate Renormalization}. However, as explained in Remark \ref{Remark: Renormalization 2},
it is crucial for our method of proof that this specific renormalization be used.
\end{remark}

\section{Intersection Local Times}
\label{Section: SILT and MILT}

In this section, we collect several results on the intersection local times of planar Brownian motions
and bridges, which are essential in Steps 2 and 3 of the outline of proof presented in
Section \ref{Section: Outline of Proof}. The main references for the known results in this section
are the monograph \cite{ChenBook}, which concerns Brownian motions, and Section 3 of the recent paper \cite{Matsuda},
which concerns Brownian bridges.

\subsection{Outline of Results}

\begin{definition}
\label{Definition: Approximate SILT and MILT}
Recall \eqref{Equation: Gaussian Kernel}.
Let $Z$ and $Z_1,Z_2$ be stochastic processes taking values in $\mbb R^2$.
For every $\eps>0$ and bounded Borel measurable set $A\subset[0,\infty)^2$, define
the approximate self-intersection local time (SILT) of $Z$ on $A$ as
\[\be^\eps_A(Z)=\int_Ap_\eps\big(Z(r_1)-Z(r_2)\big)\d r,\]
and the approximate mutual intersection local time (MILT) of $Z_1$ and $Z_2$ on $A$ as
\[\al^\eps_A(Z_1,Z_2)=\int_{A}p_\eps\big(Z_1(r_1)-Z_2(r_2)\big)\d r.\]
\end{definition}

\begin{remark}
\label{Remark: Alternate MILT}
Given that $p_\eps=p_{\eps/2}*p_{\eps/2}$ (see \eqref{Equation: General xi eps Covariance}), we can equivalently write
\[\al^\eps_A(Z_1,Z_2)=\int_{\mbb R}\int_{A}p_{\eps/2}\big(Z_1(r_1)-y\big)p_{\eps/2}\big(Z_2(r_2)-y\big)\d r\dd y.\]
In particular, when $A=A_1\times A_2$, this reduces to
\begin{align}
\label{Equation: Alternate MILT}
\al^\eps_A(Z_1,Z_2)=\int_{\mbb R^2}\prod_{i=1}^2\int_{A_i}p_{\eps/2}\big(Z_i(r)-y\big)\dd r\d y.
\end{align}
\end{remark}

Since $p_\eps\to\de_0$ as $\eps\to0$, the SILT and MILT approximate functionals that
 count how often the path of a stochastic process $Z$ intersects itself,
or how often the paths of two stochastic processes $Z_1$ and $Z_2$ intersect each other.
We now discuss the aspects of the theories of SILTs and MILTs that we need in this paper:

\subsubsection{Self-Intersection}

The construction and basic properties of the SILT that we need in
this paper are summarized in the following
results, most of which were proved in \cite{ChenBook,GuXu,Matsuda}
(see each statement for precise references):

\begin{notation}
Given any set $A\subset\mbb R^2$, let $A_{\leq}=\big\{(x,y)\in A:x\leq y\big\}.$
\end{notation}

\begin{theorem}[{\cite[Theorems 2.3.2 and 2.4.1]{ChenBook} and \cite[Theorem 3.7]{Matsuda}}]
\label{Theorem: Existence of SILT}
Let $x\in\mbb R^2$ and $t>0$ be fixed,
and suppose that $Z$ is either $B^x$ or $B^{x,x}_t$.
\begin{enumerate}[(1)]
\item Let $0\leq a<b\leq c<d\leq t$.
There exists a random variable $\be_{[a,b]\times[c,d]}(Z)$ such that
\[\be_{[a,b]\times[c,d]}(Z)=\lim_{\eps\to0}\be^\eps_{[a,b]\times[c,d]}(Z)
\qquad\text{in probability.}\]
\item If $0\leq a<b\leq t$, then
there exists a random variable $\ga_{[a,b]^2_\leq}(Z)$ such that
\[\ga_{[a,b]^2_\leq}(Z)=\lim_{\eps\to0}\Big(\be^\eps_{[a,b]^2_\leq}(Z)-\mbf E\Big[\be^\eps_{[a,b]^2_\leq}(Z)\Big]\Big)
\qquad\text{in probability.}\]
\end{enumerate}
\end{theorem}

\begin{notation}
\label{Notation: SILT Shorthand}
For any $t>0$, we use the following shorthands:
\[\be^\eps_t(Z)=\be^\eps_{[0,t]_{\leq}^2}(Z)
\qquad\text{and}\qquad
\ga_t(Z)=\ga_{[0,t]_{\leq}^2}(Z).\]
\end{notation}

In \cite{ChenBook,Matsuda}, Theorem \ref{Theorem: Existence of SILT} is only
proved for $x=0$. That said, one easily obtains
the general statement for $x\in\mbb R^2$ by noting that
SILTs are invariant under constant shifts, i.e.,
$\be^\eps_A(Z)\deq\be^\eps_A(Z-x)$.
In addition to this, SILTs of Brownian motions and bridges
satisfy the following scaling properties:

\begin{lemma}
\label{Lemma: Scaling of Approximate SILT}
Let $t>0$ and $x\in\mbb R^2$ be fixed. For every $\eps>0$ and $A\subset[0,t]^2$,
\[\be^{\eps}_A(B^{x,x}_t)\deq t\,\be^{\eps/t}_{A/t}(B^{0,0}_1)
\qquad\text{and}\qquad
\be^{\eps}_{A}(B^x)\deq t\,\be^{\eps/t}_{A/t}(B^0).\]
In particular, if $A=[a,b]^2_{\leq}$ for $0\leq a<b\leq t$, then by Theorem \ref{Theorem: Existence of SILT}-(2),
\[\ga_A(B^{x,x}_t)\deq t\,\ga_{A/t}(B^{0,0}_1)
\qquad
\text{and}
\qquad
\ga_{A}(B^x)\deq t\,\ga_{A/t}(B^0).\]
\end{lemma}

This is proved in Section \ref{Section: Scaling of Approximate SILT}. In closing
this section, we state
two properties of SILTs that are instrumental in our calculations:

\begin{lemma}
\label{Lemma: Uniform Integrability of SILT}
For every $c>0$,
there exists $\theta_c>0$ such that
for every $t\in(0,\theta_c)$,
\begin{align}
\label{Equation: Uniform Integrability of SILT - Renormalized}
\sup_{x\in D,~\eps>0}\mbf E\left[\exp\Big(c\big(\be^\eps_t(Z)-\mbf E[\be^\eps_t(Z)]\big)\Big)\right]
<\infty
\end{align}
and
\begin{align}
\label{Equation: Uniform Integrability of SILT - Non-Renormalized}
\sup_{x\in D,~\eps>0}\mbf E\left[\exp\Big(c\be^\eps_{[0,t/2]\times[t/2,t]}(Z)\Big)\right]<\infty,
\end{align}
where $Z$ is either $B^{x,x}_t$ or $B^x$.
\end{lemma}

The estimate
\eqref{Equation: Uniform Integrability of SILT - Renormalized} was proved in \cite[Lemma A.1]{GuXu} and \cite[Theorem 3.7]{Matsuda}.
We did not find \eqref{Equation: Uniform Integrability of SILT - Non-Renormalized} in the literature, so we prove it in Section \ref{Section: Uniform Integrability of SILT - Non-Renormalized} below.

\begin{remark}
\eqref{Equation: Uniform Integrability of SILT - Renormalized},
Theorem \ref{Theorem: Existence of SILT}-(2), and Lemma \ref{Lemma: Scaling of Approximate SILT}
immediately imply the following:
For every $c>0$,
there exists $\theta_c>0$ such that
for every $t\in(0,\theta_c)$,
\begin{align}
\label{Equation: Exponential Moment SILT}
\sup_{x\in D}\mbf E\left[\exp\big(c\ga_t(Z)\big)\right]<\infty,
\end{align}
where $Z$ is either $B^{x,x}_t$ or $B^x$.
\end{remark}

We end the section on SILTs with the following calculation,
which explains why subtracting $\be^\eps_t(Z)$'s mean
is necessary to get a nontrivial limit in Theorem \ref{Theorem: Existence of SILT}-(2):

\begin{lemma}[{\cite[Lemma 3.4 and Remark 3.5]{Matsuda}}]
\label{Lemma: Smoothed Expectation of SILT}
Let $t>0$ be fixed.
\begin{enumerate}[(1)]
\item If $0\leq a<b\leq t$, then
as $\eps\to0$, one has
\[\mbf E\big[\be^\eps_{[a,b]^2_\leq}(B^{0,0}_t)\big]
=\frac{b-a}{2\pi}\big(\log(1/\eps)+\log t\big)+\frac1{2\pi}\int_0^{b-a}\log r-\log(t-r)\d r+o(1)\]
and
\[\mbf E\big[\be^\eps_{[a,b]^2_{\leq}}(B^0)\big]
=\tfrac{b-a}{2\pi}\big(\log(1/\eps)+\log(b-a)-1\big)+o(1).\]
\item If $0\leq a<b\leq c<d\leq t$, then as $\eps\to0$, one has
\[\mbf E\big[\be^\eps_{[a,b]\times[c,d]}(B^{0,0}_t)\big]=\frac1{2\pi}\int_{c-a}^{d-a}\log r-\log(t-r)\d r
-\frac1{2\pi}\int_{c-b}^{d-b}\log r-\log(t-r)\d r+o(1).\]
\end{enumerate}
\end{lemma}

\begin{remark}
\label{Remark: Renormalization 2}
Following-up on Remark \ref{Remark: Renormalization 1},
the presence of $\log(1/\eps)$ in Lemma \ref{Lemma: Smoothed Expectation of SILT}-(1)
explains the necessity of using a renormalization constant of the form
\eqref{Equation: Alternate Renormalization}. Moreover, the presence of
$\tfrac{b-a}{2\pi}\log t$ and $\tfrac{b-a}{2\pi}\log(b-a)$ explains the presence
of the logarithmic terms in \eqref{Equation: Main Asymptotic 1}
and \eqref{Equation: Main Asymptotic 2}.
See Remark \ref{Remark: Renormalization 3} for more details.
\end{remark}

\subsubsection{Mutual Intersection}

The construction and basic properties of the MILT that we need in
this paper are summarized in the following results:

\begin{theorem}
\label{Theorem: Existence of MILT}
Let $t>0$ and $x_1,x_2\in\mbb R^2$ be fixed.
Suppose that the pair $(Z_1,Z_2)$ is either $({_1}B^{x_1},{_2}B^{x_2})$
or $({_1}B^{x_1,x_1}_t,{_2}B^{x_2,x_2}_t)$.
For any Borel measurable $A_1,A_2\subset[0,t]$,
there exists a random variable $\al_{A_1\times A_2}(Z_1,Z_2)$ such that
\[\al_{A_1\times A_2}(Z_1,Z_2)=\lim_{\eps\to0}\al^\eps_{A_1\times A_2}(Z_1,Z_2)\qquad\text{in probability}.\]
\end{theorem}

The statement involving independent Brownian motions $({_1}B^{x_1},{_2}B^{x_2})$ in
Theorem \ref{Theorem: Existence of MILT} is well-known
(e.g., \cite[Pages 41 and 42]{ChenBook}). While we expect that
the corresponding statement for independent Brownian bridges $({_1}B^{x_1,x_1}_t,{_2}B^{x_2,x_2}_t)$ is folklore (as
its proof is very similar), we did not find it in the literature.
Thus, for convenience, we provide a proof in
Section \ref{Section: Existence of MILT} below.

\begin{notation}
\label{Notation: MILT Shorthand}
For every $\eps,t>0$, we use the following shorthands:
\[\al^\eps_t(Z_1,Z_2)=\al^\eps_{[0,t]^2}(Z_1,Z_2)
\qquad\text{and}\qquad
\al_t(Z_1,Z_2)=\al_{[0,t]^2}(Z_1,Z_2).\]
\end{notation}

The following result shows that disjoint paths have a MILT of zero:

\begin{lemma}
\label{Lemma: No Intersection Means no MILT}
Let $t>0$ and $x_1,x_2\in\mbb R^2$ be fixed.
Suppose that the pair $(Z_1,Z_2)$ is either $({_1}B^{x_1},{_2}B^{x_2})$
or $({_1}B^{x_1,x_1}_t,{_2}B^{x_2,x_2}_t)$.
Given a positive $\theta>0$,
let us denote the event
\[\mf O_{t,\theta}(Z_1,Z_2)=\left\{\inf_{(r_1,r_2)\in[0,t]^2}\|Z_1(r_1)-Z_2(r_2)\|\geq\theta\right\}.\]
Almost surely, $\al_t(Z_1,Z_2)\mbf 1_{\mf O_{t,\theta}(Z_1,Z_2)}=0.$
\end{lemma}

See Section \ref{Section: No Intersection Means no MILT} for a proof.
The next result combines a scaling property and the claim that
MILT's moments are maximized if two paths start at the same location:

\begin{lemma}
\label{Lemma: MILT Scaling Property}
Let $t>0$ be fixed.
For any integer $m\geq1$,
\[\sup_{x_1,x_2\in\mbb R^2}\mbf E\big[\al_t({_1}B^{x_1,x_1}_t,{_2}B^{x_2,x_2}_t)^m\big]\leq t^m\,\mbf E\big[\al_1({_1}B^{0,0}_1,{_2}B^{0,0}_1)^m\big]<\infty\]
and
\[\sup_{x_1,x_2\in\mbb R^2}\mbf E\big[\al_t({_1}B^{x_1},{_2}B^{x_2})^m\big]\leq t^m\,\mbf E\big[\al_1({_1}B^0,{_2}B^0)^m\big]<\infty.\]
\end{lemma}

This result was proved in \cite[Proposition 2.2.6 and (2.2.24)]{ChenBook} in the
case of Brownian motion. Since we did not find the corresponding result for 
Brownian bridges in the literature, we prove it in
Section \ref{Section: MILT Scaling Property}. Lastly, we have the following uniform integrability estimate
for MILTs:

\begin{lemma}
\label{Lemma: MILT is UI}
For every $c>0$, there exists $\theta_c>0$ such that for every $t\in(0,\theta_c)$,
\[\sup_{\eps>0,~x_1,x_2\in\mbb R^2}\mbf E\left[\exp\Big(c\,\al^\eps_t(Z_1,Z_2)\Big)\right]
\leq\sup_{x_1,x_2\in\mbb R^2}\mbf E\left[\exp\Big(c\,\al_t(Z_1,Z_2)\Big)\right]<\infty,\]
where $(Z_1,Z_2)$ is either $({_1}B^{x_1},{_1}B^{x_2})$
or $({_1}B^{x_1,x_1}_t,{_2}B^{x_2,x_2}_t)$.
\end{lemma}

This result is well-known for Brownian motions (e.g.,
\cite[Theorem 2.2.9 and Equation (2.2.24)]{ChenBook}). See Section \ref{Section: MILT is UI} for a
proof of the result for Brownian bridges.

\subsection{Proof of Lemma \ref{Lemma: Scaling of Approximate SILT}}
\label{Section: Scaling of Approximate SILT}

We only prove the result for $Z=B^{x,x}_t$,
since the argument is exactly the same if $Z=B^x$.
Fix $\eps>0$.
Notice that, by Brownian scaling,
\begin{align}
\label{Equation: Brownian Scaling}
\big\{B_t^{x,x}(r):r\in[0,t]\big\}\deq\big\{x+\sqrt{t}\,B^{0,0}_1(r/t):r\in[0,t]\big\}.
\end{align}
Therefore,
\[
\be^\eps_A(B^{x,x}_t)
=\int_{A}p_\eps\big(B_{t}^{x,x}(r_1)-B_{t}^{x,x}(r_2)\big)\d r
\deq\int_{A}p_\eps\big(\sqrt tB_1^{0,0}(r_1/t)-\sqrt tB_1^{0,0}(r_2/t)\big)\d r.\]
If we define
$s_1=\frac{r_1}{t}$ and $s_2=\frac{r_2}{t}$, then $s\in A/t$
and $\dd r=t^2\dd s$; hence
\[
\be^\eps_A(B^{x,x}_t)
	\deq\int_{A/t}t^2p_\eps\big(\sqrt tB_1^{0,0}(s_1)-\sqrt tB_1^{0,0}(s_2)\big)\d s.
\]
By definition of the Gaussian kernel \eqref{Equation: Gaussian Kernel},
\[t^2p_\eps(\sqrt tx)
=t^2\frac{\mr e^{-|x|^2/2(\eps/t)}}{2\pi\eps}=tp_{\eps/t}(x),\]
hence
\[
\be^\eps_A(B^{x,x}_t)
	\deq t\int_{A/t}p_{\eps/t}\big(B_t^{0,0}(s_1)-B_t^{0,0}(s_2)\big)\d s=t\,\be^{\eps/t}_{A/t}(B^{0,0}_1).
\]

\subsection{Proof of Theorem \ref{Theorem: Existence of MILT}}
\label{Section: Existence of MILT}

Let $t>0$, $x_1,x_2\in\mbb R^2$, and
$A_1,A_2\subset[0,t]$ be fixed.
By \cite[Pages 41 and 42]{ChenBook}, we only need to prove the result for
of $(Z_1,Z_2)=({_1}B^{x_1,x_1}_t,{_2}B^{x_2,x_2}_t)$.
For this, we follow along the steps of the proof of
\cite[Theorem 2.2.3]{ChenBook}.

For sake of readability, for the remainder of this proof we denote
\[X_\eps=\al^\eps_{A_1\times A_2}({_1}B^{x_1,x_1}_t,{_2}B^{x_2,x_2}_t).\]
Given that the $L^2$-space of square-integrable random variables is complete,
and that convergence in $L^2$ implies convergence in probability,
it suffices to show that
\[\lim_{\eps,\tilde\eps\to0}\mbf E\left[(X_\eps-X_{\tilde\eps})^2\right]=0.\]
If we expand the square above, then we notice that it is enough to show that there exists a finite number $\zeta>0$ such that
\[\lim_{\eps,\tilde\eps\to0}\mbf E[X_\eps\,\,X_{\tilde\eps}]=\zeta.\]
Instead of only proving this, we will establish the following much more general result,
since we need the latter in the proofs of other lemmas: For every integer $m\geq1$,
let $\Si_m$ denote the set of permutations of $\{1,2,\ldots,m\}$,
and given $K\subset\mbb[0,\infty)^m$, let us denote $K_<=\{r=(r_1,\ldots,r_m)\in K:r_1<r_2<\cdots<r_m\}.$
Then,
\begin{align}
\label{Equation: Moments of MILT Prelimit}
\lim_{\eps_1,\ldots,\eps_m\to0}\mbf E[X_{\eps_1}\cdots X_{\eps_m}]
=\int_{(\mbb R^2)^m}\msf g^m_{t,A_1}(x_1,z)\msf g^m_{t,A_2}(x_2,z)\d z,
\end{align}
where for any $t>0$, $x\in\mbb R^2$, $m\geq1$ and $z\in(\mbb R^2)^m$, and $A\subset[0,t]$, we define
\begin{align}
\label{Equation: Moments of MILT Function}
\msf g^m_{t,A}(x,z)=\sum_{\si\in\Si_m}\int_{(A^m)_<}\frac{p_{r_1}(z_{\si(1)}-x)\cdot\prod_{j=2}^mp_{r_{j}-r_{j-1}}(z_{\si(j)}-z_{\si(j-1)})\cdot p_{t-r_m}(x-z_{\si(m)})}{p_t(0)}\d r.
\end{align}

Toward this end, we begin by calculating the moments of the form $\mbf E[X_{\eps_1}\cdots X_{\eps_m}]$
for fixed $\eps_j>0$. For this purpose, if we use \eqref{Equation: Alternate MILT} to write the MILT, then
the independence of ${_i}B^{x_i,x_i}_t$ for $i=1,2$ and Tonelli's theorem yields
\begin{align*}
\mbf E[X_{\eps_1}\cdots X_{\eps_m}]
=\int_{(\mbb R^2)^m}\prod_{i=1}^2\int_{A_i^m}\mbf E\left[\prod_{j=1}^mp_{\eps_j/2}\big({_i}B^{x_i,x_i}_t(r_j)-y_j\big)\right]\d r \d y.
\end{align*}
The expectation above is easier to calculate if the $r_j$'s are ordered thanks to the
following explicit calculation:
If $0<r_1<r_2<\cdots<r_m<t$, then for any $z\in(\mbb R^2)^m$,
\begin{align*}
\mbf P\left[B^{x,x}_t(r_1)\in\dd z_1,\ldots,B^{x,x}_t(r_m)\in\dd z_m\right]
=\frac{\mbf P\left[B^x(r_1)\in\dd z_1,\ldots,B^x_t(r_m)\in\dd z_m,B^x(t)\in\dd x\right]}{\mbf P\big[B^x(t)\in\dd x\big]}\\
=\frac{p_{r_1}(z_{1}-x)\cdot\prod_{j=2}^mp_{r_{j}-r_{j-1}}(z_{j}-z_{j-1})\cdot p_{t-r_m}(x-z_m)}{p_t(0)}.
\end{align*}
Toward this end, for any permutation $\si\in\Si_m$
and Borel measurable $K\subset\mbb R^m$, let
\[K_<\circ\si=\big\{r=(r_1,r_2,\ldots,r_m)\in K:r_{\si(1)}<r_{\si(2)}<\cdots<r_{\si(m)}\big\}.\]
Note that we can write any Borel measurable $K$ as
\[K=\left(\bigcup_{\si\in\Si_m}K_<\circ\si\right)\cup L,\]
where the union is disjoint and $L$ has Lebesgue measure zero.
Thus,
\begin{align*}
\mbf E[X_{\eps_1}\cdots X_{\eps_m}]
=\int_{(\mbb R^2)^m}\prod_{i=1}^2\left(\sum_{\si\in\Si_m}\int_{(A_i^p)_<\circ\si}\mbf E\left[\prod_{j=1}^mp_{\eps_{\si(j)}/2}\big({_i}B^{x_i,x_i}_t(r_{\si(j)})-y_{\si(j)}\big)\right]\d r\right)\dd y.
\end{align*}
If we now apply the change of variables $r_j=r_{\si(j)}$, then we get
\begin{align*}
\mbf E[X_{\eps_1}\cdots X_{\eps_m}]
=\int_{(\mbb R^2)^m}\prod_{i=1}^2\left(\sum_{\si\in\Si_m}\int_{(A_i^p)_<}\mbf E\left[\prod_{j=1}^mp_{\eps_{\si(j)}/2}\big({_i}B^{x_i,x_i}_t(r_j)-y_{\si(j)}\big)\right]\d r\right)\dd y.
\end{align*}
With the components $r_j$ ordered, we can now calculate the expectation inside this integral as
\begin{multline*}
\mbf E\left[\prod_{j=1}^mp_{\eps_{\si(j)}/2}\big({_i}B^{x_i,x_i}_t(r_j)-y_{\si(j)}\big)\right]
=\int_{(\mbb R^2)^m}\left(\prod_{j=1}^mp_{\eps_j/2}(z_j-y_j)\right)\\
\cdot\frac{p_{r_1}(z_{\si(1)}-x_i)\cdot\prod_{j=2}^mp_{r_{j}-r_{j-1}}(z_{\si(j)}-z_{\si(j-1)})\cdot p_{t-r_m}(x_i-z_{\si(m)})}{p_t(0)}\d z
\end{multline*}
(note that we have also used $\prod_{j=1}^mp_{\eps_{\sigma(j)}/2}(z_{\sigma(j)}-y_{\sigma(j)})=\prod_{j=1}^mp_{\eps_j/2}(z_j-y_j)$).
Plugging this back into our formula for $\mbf E[X_{\eps_1}\cdots X_{\eps_m}]$, we are led to
\begin{multline*}
\mbf E[X_{\eps_1}\cdots X_{\eps_m}]
=\int_{(\mbb R^2)^m}
\prod_{i=1}^2\Bigg(\sum_{\si\in\Si_m}\int_{(A_i^p)_<}\int_{(\mbb R^2)^m}
\left(\prod_{j=1}^mp_{\eps_j/2}(z_j-y_j)\right)\\
\cdot\frac{p_{r_1}(z_{\si(1)}-x_i)\cdot\prod_{j=2}^mp_{r_{j}-r_{j-1}}(z_{\si(j)}-z_{\si(j-1)})\cdot p_{t-r_m}(x_i-z_{\si(m)})}{p_t(0)}\d z\dd r\Bigg)\dd y.
\end{multline*}
At this point, if we apply Tonelli's theorem to interchange the $\dd z$ integral
with the $\dd r$ integral and the sum over $\Si_m$,
and then combine this with the fact that $p_{\eps_j/2}(z_j-y_j)$ does not depend on either $r$ or $\si$,
then we obtain
\begin{multline*}
\mbf E[X_{\eps_1}\cdots X_{\eps_m}]
=\int_{(\mbb R^2)^m}
\prod_{i=1}^2\int_{(\mbb R^2)^m}
\left(\prod_{j=1}^mp_{\eps_j/2}(z_j-y_j)\right)\Bigg(\sum_{\si\in\Si_m}\int_{(A_i^p)_<}\cdots\\
\frac{p_{r_1}(z_{\si(1)}-x_i)\cdot\prod_{j=2}^mp_{r_{j}-r_{j-1}}(z_{\si(j)}-z_{\si(j-1)})\cdot p_{t-r_m}(x_i-z_{\si(m)})}{p_t(0)}\d r\dd z\Bigg)\dd y.
\end{multline*}
Recalling the definition of $\msf g^m_{t,A}(x,z)$ in \eqref{Equation: Moments of MILT Function},
this can be rewritten as
\begin{align}
\label{Equation: MILT Pre-Moment}
\mbf E[X_{\eps_1}\cdots X_{\eps_m}]=
\int_{(\mbb R^2)^m}\prod_{i=1}^2\left(\int_{(\mbb R^2)^m}\left(\prod_{j=1}^mp_{\eps_j/2}(z_j-y_j)\right)\msf g^m_{t,A_i}(x_i,z)\d z\right)\dd y.
\end{align}
By standard results on the $L^p$ convergence of products of functions convolved with a smooth mollifier
(e.g., \cite[Lemma 2.2.2 and Page 32]{ChenBook}), \eqref{Equation: Moments of MILT Prelimit} will follow if we show that
the function $z\mapsto\msf g^m_{t,A_1}(x_1,z)\msf g^m_{t,A_2}(x_2,z)$ is integrable, i.e.:
\begin{align}
\label{Equation: Moments of MILT Integrable}
\int_{(\mbb R^2)^m}\msf g^m_{t,A_1}(x_1,z)\msf g^m_{t,A_2}(x_2,z)\d z<\infty.
\end{align}

By definition of $\msf g^m_{t,A}(x,z)$ in \eqref{Equation: Moments of MILT Function},
it is clear that $\msf g^m_{t,A}(x,z)\leq\msf g^m_{t,[0,t]}(x,z)$ if $A\subset[0,t]$.
Thus, by H\"older's inequality,
\begin{align}
\label{Equation: Moments of MILT Integrable 1}
\int_{(\mbb R^2)^m}\msf g^m_{t,A_1}(x_1,z)\msf g^m_{t,A_2}(x_2,z)\d z
\leq\left(\prod_{i=1}^2\int_{(\mbb R^2)^m}\msf g^m_{t,[0,t]}(x_i,z)^2\d z\right)^{1/2}.
\end{align}
Then, if we apply the change of variables $(z_1,\ldots,z_m)\mapsto(z_1+x_i,\ldots,z_m+x_i)$ for $i=1,2$ above, it is clear from
\eqref{Equation: Moments of MILT Function} that \eqref{Equation: Moments of MILT Integrable 1} can be reformulated into
\begin{align}
\label{Equation: Moments of MILT Integrable 2}
\int_{(\mbb R^2)^m}\msf g^m_{t,A_1}(x_1,z)\msf g^m_{t,A_2}(x_2,z)\d z
\leq\int_{(\mbb R^2)^m}\msf g^m_{t,[0,t]}(0,z)^2\d z.
\end{align}
Given any two sets $A,\tilde A\subset[0,t]$ whose intersection has Lebesgue measure zero, it is clear from \eqref{Equation: Moments of MILT Function} that
$\msf g^m_{t,A}(0,z)+\msf g^m_{t,\tilde A}(0,z)=\msf g^m_{t,A\cup\tilde A}(0,z)$.
If we apply this with $A=[0,t/2]$ and $\tilde A=[t/2,t]$, then we get
\[\msf g^m_{t,[0,t]}(0,z)^2=\msf g^m_{t,[0,t/2]}(0,z)^2+2\msf g^m_{t,[0,t/2]}(0,z)\msf g^m_{t,[t/2,t]}(0,z)+\msf g^m_{t,[t/2,t]}(0,z)^2.\]
If we apply the change of variables $(r_1,\ldots,r_m)\mapsto(t-r_1,\ldots,t-r_m)$ in \eqref{Equation: Moments of MILT Function},
then we get that $\msf g^m_{t,[t/2,t]}(0,z)=\msf g^m_{t,[0,t/2]}(0,z)$. Thus, we can rewrite \eqref{Equation: Moments of MILT Integrable 2}
as
\begin{align}
\label{Equation: Moments of MILT Integrable 3}
\int_{(\mbb R^2)^m}\msf g^m_{t,A_1}(x_1,z)\msf g^m_{t,A_2}(x_2,z)\d z
\leq4\int_{(\mbb R^2)^m}\msf g^m_{t,[0,t/2]}(0,z)^2\d z.
\end{align}
If $0\leq r_m\leq t/2$, then
\[\sup_{z\in\mbb R^2}\frac{p_{t-r_m}(z)}{p_t(0)}\leq\frac{2\pi t}{2\pi(t/2)}=2.\]
If we apply this to \eqref{Equation: Moments of MILT Function},
then we get from \eqref{Equation: Moments of MILT Integrable 3} that
\begin{align}
\label{Equation: Moments of MILT Integrable 4}
\int_{(\mbb R^2)^m}\msf g^m_{t,A_1}(x_1,z)\msf g^m_{t,A_2}(x_2,z)\d z
\leq8\int_{(\mbb R^2)^m}\msf f^m_{[0,t/2]}(z)^2\d z.
\end{align}
where for any $m\geq1$, $z\in(\mbb R^2)^m$, and bounded Borel set $A\subset[0,\infty)$, we define
\begin{align}
\label{Equation: Moments of MILT Function for BM}
\msf f^m_{A}(z)=\sum_{\si\in\Si_m}\int_{(A^m)_<}p_{r_1}(z_{\si(1)})\cdot\prod_{j=2}^mp_{r_{j}-r_{j-1}}(z_{\si(j)}-z_{\si(j-1)})\d r.
\end{align}
With this inequality in hand, we can now prove \eqref{Equation: Moments of MILT Integrable}---hence
conclude the proof of Theorem \ref{Theorem: Existence of MILT}---by
noting the following: According to \cite[Theorem 2.2.3]{ChenBook},
for any bounded Borel sets $A_1,A_2\subset[0,\infty)$, one has
\begin{align}
\label{Equation: Moments of Brownian Motion MILT}
\mbf E\Big[\al_{A_1\times A_2}\big({_1}B^0,{_2}B^0\big)^m\Big]=\int_{(\mbb R^2)^m}\msf f^m_{A_1}(z)\msf f^m_{A_2}(z)\d z<\infty.
\end{align}

\subsection{Proof of Lemma \ref{Lemma: No Intersection Means no MILT}}
\label{Section: No Intersection Means no MILT}

By Theorem \ref{Theorem: Existence of MILT}, it suffices to show that
\[\lim_{\eps\to0}\al^\eps_t(Z_1,Z_2)\mbf 1_{\mf O_{t,\theta}(Z_1,Z_2)}\qquad\text{almost surely.}\]
Toward this end, for every $\eps>0$,
Definition \ref{Definition: Approximate SILT and MILT} implies that

\begin{multline*}
	\al^\eps_t(Z_1,Z_2)\mbf 1_{\mf O_{t,\theta}(Z_1,Z_2)}
	=
	\int_{[0,t]^2}
	p_\eps\left(Z_1(r_1),Z_2(r_2)\right)
	\mbf 1_{\mf O_{t,\theta}(Z_1,Z_2)}
	\d r\\
	=
	\int_{[0,t]^2}
	\frac{\mr e^{-\left\lVert Z_1(r_1)-Z_2(r_2)\right\rVert^2/2\eps}}{2\pi\eps}
	\mbf 1_{\mf O_{t,\theta}(Z_1,Z_2)}
	\d r
	\leq
	\int_{[0,t]^2}
	\frac{\mr e^{-\theta^2/2\eps}}{2\pi\eps}
	\mbf 1_{\mf O_{t,\theta}(Z_1,Z_2)}
	\d r
	\leq	t^2\frac{\mr e^{-\theta^2/2\eps}}{2\pi\eps}.
\end{multline*}
This vanishes as $\eps\to0$, thus concluding the proof.

\subsection{Proof of Lemma \ref{Lemma: MILT Scaling Property}}
\label{Section: MILT Scaling Property}

Thanks to \cite[Proposition 2.2.6 and (2.2.24)]{ChenBook}, we only need to prove the result for Brownian bridges.
The main step of the proof is to establish the identity
\begin{align}
\label{Equation: Moments of MILT}
\mbf E\big[\al_{A_1\times A_2}({_1}B^{x_1,x_1}_t,{_2}B^{x_2,x_2}_t)^m\big]
=\int_{(\mbb R^2)^m}\msf g^m_{t,A_1}(x_1,z)\msf g^m_{t,A_2}(x_2,z)\d z,
\end{align}
where we recall the definition of $\msf g^m_{t,A}(x,z)$ in \eqref{Equation: Moments of MILT Function}.
Indeed, once this is done, then we get that
\[\mbf E\big[\al_t({_1}B^{x_1,x_1}_t,{_2}B^{x_2,x_2}_t)^m\big]
\leq
\mbf E\big[\al_t({_1}B^{0,0}_t,{_2}B^{0,0}_t)^m\big]\]
from \eqref{Equation: Moments of MILT Integrable 2} (with $A_1=A_2=[0,t]$). Next, we get
\[\mbf E\big[\al_t({_1}B^{0,0}_t,{_2}B^{0,0}_t)^m\big]=t^m\mbf E\big[\al_t({_1}B^{0,0}_1,{_2}B^{0,0}_1)^m\big]<\infty\]
by combining \eqref{Equation: Moments of MILT Integrable} with the scaling property
$p_{st}(z)=t^{-1}p_s(z/t^{1/2})$ and changes of variables in the $\dd r$ integral in \eqref{Equation: Moments of MILT Function}
and the $\dd z$ integral in \eqref{Equation: Moments of MILT}.

It now only remains to prove \eqref{Equation: Moments of MILT}.
We know from \eqref{Equation: Moments of MILT Prelimit} that
\begin{align}
\lim_{\eps\to0}\mbf E\big[\al^\eps_{A_1\times A_2}({_1}B^{x_1,x_1}_t,{_2}B^{x_2,x_2}_t)^m\big]
=\int_{(\mbb R^2)^m}\msf g^m_{t,A_1}(x_1,z)\msf g^m_{t,A_2}(x_2,z)\d z.
\end{align}
If we combine this with Theorem \ref{Theorem: Existence of MILT}, then by
the Vitali convergence theorem, it suffices to prove
that the sequence
$\big\{\al^\eps_{A_1\times A_2}({_1}B^{x_1,x_1}_t,{_2}B^{x_2,x_2}_t)^m:\eps>0\big\}$
is uniformly integrable. For this purpose,
we recall a special case of the well-known de-la-Vall\'ee-Poussin criterion for uniform integrability (u.i.):
\begin{align}
\label{Equation: Poussin}
\text{If }\sup_{\theta\in\Theta}\mbf E\big[|X_\theta|^q\big]<\infty\text{ for some $q>1$, then }\{X_\theta:\theta\in\Theta\}\text{ is u.i.}
\end{align}
Thus, it is enough to show that for every $m\geq1$,
\begin{align*}
\sup_{\eps>0}\mbf E\big[\al^\eps_{A_1\times A_2}({_1}B^{x_1,x_1}_t,{_2}B^{x_2,x_2}_t)^m\big]<\infty.
\end{align*}
To see this, if we use \eqref{Equation: MILT Pre-Moment} (with $\eps_j=\eps$ for $j=1,\ldots,m$),
together with an application of Tonelli's theorem to expand the product over $i$ as two distinct
integrals over $z_j$ and $\tilde z_j$, then we get
\begin{multline*}
\mbf E\big[\al^\eps_{A_1\times A_2}({_1}B^{x_1,x_1}_t,{_2}B^{x_2,x_2}_t)^m\big]
=\int_{((\mbb R^2)^m)^2}\msf g^m_{t,A_1}(x_1,z)\msf g^m_{t,A_2}(x_2,\tilde z)\\
\cdot\left(\int_{(\mbb R^2)^m}\prod_{j=1}^mp_{\eps/2}(z_j-y_j)p_{\eps/2}(\tilde z_j-y_j)\d y\right)\dd z\dd\tilde z.
\end{multline*}
If we use the semigroup property of the Gaussian kernel, then this becomes
\[\mbf E\big[\al^\eps_{A_1\times A_2}({_1}B^{x_1,x_1}_t,{_2}B^{x_2,x_2}_t)^m\big]
=\int_{((\mbb R^2)^m)^2}\msf g^m_{t,A_1}(x_1,z)\msf g^m_{t,A_2}(x_2,\tilde z)\prod_{j=1}^mp_{\eps}(z_j-\tilde z_j)\d z\dd\tilde z.\]
By Young's convolution inequality and the fact that $p_\eps$ is a probability density,
we then get that
\begin{multline}
\label{Equation: MILT UI Precursor}
\sup_{\eps>0}\mbf E\big[\al^\eps_{A_1\times A_2}({_1}B^{x_1,x_1}_t,{_2}B^{x_2,x_2}_t)^m\big]\\
\leq\left(\int_{(\mbb R^2)^m}\msf g^m_{t,A_1}(x_1,z)^2\d z\right)^{1/2}\left(\int_{(\mbb R^2)^m}\msf g^m_{t,A_2}(x_2,z)^2\d z\right)^{1/2};
\end{multline}
this upper bound is finite by \eqref{Equation: Moments of MILT Integrable}, thus concluding the proof.

\subsection{Proof of Lemma \ref{Lemma: MILT is UI}}
\label{Section: MILT is UI}

By \cite[Theorem 2.2.9 and (2.2.24)]{ChenBook}, we only need to prove the result for Brownian bridges.
If we combine an application of \eqref{Equation: MILT UI Precursor} (assuming that $A_1=A_2=[0,t]$)
with an application of the same change of variables used in \eqref{Equation: Moments of MILT Integrable 1} and
\eqref{Equation: Moments of MILT Integrable 2} (which allows to set $x_1=x_2=0$),
then we get
\begin{multline*}
\sup_{\eps>0,~x_1,x_2\in\mbb R^2}\mbf E\big[\al^\eps_{t}({_1}B^{x_1,x_1}_t,{_2}B^{x_2,x_2}_t)^m\big]\\
\leq\sup_{x_1,x_2\in\mbb R^2}\left(\int_{(\mbb R^2)^m}\msf g^m_{t,[0,t]}(x_1,z)^2\d z\right)^{1/2}\left(\int_{(\mbb R^2)^m}\msf g^m_{t,[0,t]}(x_2,z)^2\d z\right)^{1/2}\\
=\int_{(\mbb R^2)^m}\msf g^m_{t,[0,t]}(0,z)^2\d z
=\mbf E\big[\al_{t}({_1}B^{0,0}_t,{_2}B^{0,0}_t)^m\big].
\end{multline*}
Since this holds for all integers $m\geq1$, using the moment generating function of the MILT
yields
\[\sup_{\eps>0,~x_1,x_2\in\mbb R^2}\mbf E\left[\exp\Big(c\,\al^\eps_t({_1}B^{x_1,x_1}_t,{_2}B^{x_2,x_2}_t)\Big)\right]
\leq\mbf E\left[\exp\Big(c\,\al_t({_1}B^{0,0}_t,{_2}B^{0,0}_t)\Big)\right]\]
whenever the expectation on the right-hand side of the
above inequality is finite.
This is the case for small enough $t>0$ thanks to
\eqref{Equation: Moments of MILT Integrable 4},
\eqref{Equation: Moments of Brownian Motion MILT},
and the statement of Lemma \ref{Lemma: MILT is UI} in the case of Brownian motions.
In order to complete the proof of Lemma \ref{Lemma: MILT is UI}, it only remains to show that
\[\sup_{x_1,x_2\in\mbb R^2}\mbf E\left[\exp\Big(c\,\al_t({_1}B^{x_1,x_1}_t,{_2}B^{x_2,x_2}_t)\Big)\right]=\mbf E\left[\exp\Big(c\,\al_t({_1}B^{0,0}_t,{_2}B^{0,0}_t)\Big)\right].\]
Clearly the supremum on the left-hand side is greater or equal than the
expectation on the right-hand side (the latter of which consists of choosing $x_1=x_2=0$); the reverse inequality follows
from \eqref{Equation: Moments of MILT Integrable 2} (with $A_1=A_2=[0,t]$) and \eqref{Equation: Moments of MILT}.

\subsection{Proof of \eqref{Equation: Uniform Integrability of SILT - Non-Renormalized}}
\label{Section: Uniform Integrability of SILT - Non-Renormalized}

By \cite[Proposition 2.3.4]{ChenBook}, we note that
\[\be^{\eps}_{[0,t/2]\times[t/2,t]}(B^x)\deq\al^\eps_{[0,t/2]^2}({_1}B^0,{_2}B^0).\]
Thus, \eqref{Equation: Uniform Integrability of SILT - Non-Renormalized}
in the case where $Z=B^x$ follows from Lemma \ref{Lemma: MILT is UI}.
We now consider the case of Brownian bridge. By Lemma
\ref{Lemma: Scaling of Approximate SILT},
\[\be^{\eps}_{[0,t/2]\times[t/2,t]}(B^{x,x}_t)\deq t\,\be^{\eps/t}_{[0,1/2]\times[1/2,1]}(B^{0,0}_1),\]
and so it suffices to show that there exists $\theta_c$ such that for every $t\in(0,\theta_c)$, one has
\[\sup_{\eps>0}\mbf E\left[\exp\Big(ct\be^\eps_{[0,1/2]\times[1/2,1]}(B^{0,0}_1)\Big)\right]<\infty.\]
Since $A\mapsto\be^\eps_A(Z)$ is a measure
and $\be^\eps_A(Z)=0$ when $A$ has Lebesgue measure zero (both of which
clearly follow from Definition \ref{Definition: Approximate SILT and MILT}),
for any $\de\in(0,1/4)$, we can write $\be^\eps_{[0,1/2]\times[1/2,1]}(B^{0,0}_1)$ as the sum
\[\be^\eps_{[0,1/2]\times[1/2,1-\de]}(B^{0,0}_1)
+\be^\eps_{[0,\de]\times[1-\de,1]}(B^{0,0}_1)
+\be^\eps_{[\de,1/2]\times[1-\de,1]}(B^{0,0}_1).\]
Thus, by H\"older's inequality, it suffices to show that
there exists $\theta_c>0$ such that for every $t\in(0,\theta_c)$,
\begin{align}
\label{Equation: Uniform Integrability of SILT - Non-Renormalized I}
\sup_{\eps>0}\mbf E\Big[\exp\Big(3ct\be^\eps_{[0,\de]\times[1-\de,1]}(B^{0,0}_1)\Big)\Big]<\infty,
\end{align}
\begin{align}
\label{Equation: Uniform Integrability of SILT - Non-Renormalized II}
\sup_{\eps>0}\mbf E\Big[\exp\Big(3ct\be^\eps_{[\de,1/2]\times[1-\de,1]}(B^{0,0}_1)\Big)\Big]<\infty,
\end{align}
and
\begin{align}
\label{Equation: Uniform Integrability of SILT - Non-Renormalized III}
\sup_{\eps>0}\mbf E\Big[\exp\Big(3ct\be^\eps_{[0,1/2]\times[1/2,1-\de]}(B^{0,0}_1)\Big)\Big]<\infty.
\end{align}
In the proof of \cite[Theorem 3.7 (iii)]{Matsuda}, it is shown that
\eqref{Equation: Uniform Integrability of SILT - Non-Renormalized I}
and \eqref{Equation: Uniform Integrability of SILT - Non-Renormalized II}
holds 
(more specifically, $[0,\de]\times[1-\de,1]$ is the set denoted by $A_4$ in \cite{Matsuda},
and $[\de,1/2]\times[1-\de,1]$ is contained in the set denoted by $A_3$ in \cite{Matsuda}).
As for \eqref{Equation: Uniform Integrability of SILT - Non-Renormalized III}, if we let $\mbb P_\de$ denote
the law of the Brownian motion $\big(B^0(s):s\in[0,1-\de]\big)$, then under the tilted measure
\[\dd\tilde{\mbb P}_\de=\tfrac{1}{\de}\exp\left(-\tfrac{|B(1-\de)|^2}{2\de}\right)\d\mbb P_\de,\]
that process has the law $\big(B^{0,0}_1(s):s\in[0,1-\de]\big)$
(see, e.g., \cite[Lemma 3.1]{Nakao} for the details). Consequently,
\[\sup_{\eps>0}\mbf E\Big[\exp\Big(3ct\be^\eps_{[0,1/2]\times[1/2,1-\de]}(B^{0,0}_1)\Big)\Big]
\leq
\tfrac{1}{\de}\sup_{\eps>0}\mbf E\Big[\exp\Big(3ct\be^\eps_{[0,1/2]\times[1/2,1-\de]}(B^0)\Big)\Big].\]
We therefore obtain \eqref{Equation: Uniform Integrability of SILT - Non-Renormalized III}
by the statement of \eqref{Equation: Uniform Integrability of SILT - Non-Renormalized} for Brownian motions.

\section{Theorem \ref{Theorem: Main} Step 2: Feynman-Kac Formulas}
\label{Section: Step 2}

\subsection{Outline}

In this section, we provide the main tool that is used to calculate the expectation and variance asymptotics
in Theorem \ref{Theorem: Main}, namely:

\begin{proposition}
\label{Proposition: Feynman-Kac for T and M}
For every $\ka>0$,
there exists a constant $\th_\ka>0$ such that
for every $t\in(0,\th_\ka)$, the following holds:
\begin{align}
\label{Equation: Feynman-Kac for T and M 1}
\lim_{\eps\to0}\mbf E\big[\msf T_{\ka,\eps}(t)\mr e^{-t\msf c_{\ka,\eps}}\big]
=\frac{\mr e^{\ka^2t\log t/2\pi}}{2\pi t}\int_D\mbf E\Big[\mbf 1_{\{\tau_D(B^{x,x}_t)>t\}}\mr e^{\ka^2\ga_t(B^{x,x}_t)}\Big]\d x;
\end{align}
\begin{align}
\label{Equation: Feynman-Kac for T and M 2}
\lim_{\eps\to0}\mbf E\big[\msf M_{\ka,\eps}(t)\mr e^{-t\msf c_{\ka,\eps}}\big]
=\mr e^{\ka^2(t\log t-t)/2\pi}\int_D\mbf E\Big[\mbf 1_{\{\tau_D(B^x)>t\}}\mr e^{\ka^2\ga_t(B^x)}\Big]\d x;
\end{align}
\begin{multline}
\label{Equation: Feynman-Kac for T and M 3}
\lim_{\eps\to0}\mbf{Var}\big[\msf T_{\ka,\eps}(t)\mr e^{-t\msf c_{\ka,\eps}}\big]=\frac{\mr e^{\ka^2t\log t/\pi}}{(2\pi t)^2}\int_{D^2}\mbf E\Bigg[\mbf 1_{\cap_{i\leq2}\{\tau_D({_i}B^{x_i,x_i}_t)>t\}}\\
\cdot\mr e^{\ka^2\sum_{i=1}^2\ga_t({_i}B^{x_i,x_i}_t)}\left(\mr e^{\ka^2\al_t({_1}B^{x_1,x_1}_t,{_2}B^{x_2,x_2}_t)}-1\right)\Bigg]\d x;
\end{multline}
\begin{multline}
\label{Equation: Feynman-Kac for T and M 4}
\lim_{\eps\to0}\mbf{Var}\big[\msf M_{\ka,\eps}(t)\mr e^{-t\msf c_{\ka,\eps}}\big]=\mr e^{\ka^2(t\log t-t)/\pi}\int_{D^2}\mbf E\Bigg[\mbf 1_{\cap_{i\leq2}\{\tau_D({_i}B^{x_i})>t\}}\\
\cdot\mr e^{\ka^2\sum_{i=1}^2\ga_t({_i}B^{x_i})}\left(\mr e^{\ka^2\al_t({_1}B^{x_1},{_2}B^{x_2})}-1\right)\Bigg]\d x.
\end{multline}
\end{proposition}

Indeed, if we combine this result with Proposition \ref{Proposition: Moments Prelimit},
then we obtain Feynman-Kac formulas for the expectation and variance of $\msf T_\ka(t)$
and $\msf M_\ka(t)$. The remainder of this section is organized as follows:
In Section \ref{Section: Moments Prelimit}, we provide the proof of
Proposition \ref{Proposition: Moments Prelimit}, up to two technical lemmas
(i.e., Lemmas \ref{Lemma: Moments Prelimit UI} and \ref{Lemma: Hilbert-Schmidt Cauchy Sequence}),
which are respectively proved in Sections \ref{Section: Moments Prelimit UI} and
\ref{Section: Hilbert-Schmidt Cauchy Sequence}.
Proposition \ref{Proposition: Feynman-Kac for T and M} is proved in
Section \ref{Section: Feynman-Kac for T and M}.

\subsection{Proof of Proposition \ref{Proposition: Moments Prelimit}}
\label{Section: Moments Prelimit}

The main ingredient in the proof of Proposition \ref{Proposition: Moments Prelimit}
is the following claim: For every $\ka>0$ there exists $\th_\ka>0$ such that
for every $t\in(0,\th_\ka)$, one has
\begin{align}
\label{Equation: Moment Formula 1 Limits in Probability}
\lim_{\eps\to0}\msf T_{\ka,\eps}(t)\mr e^{-t\msf c_{\ka,\eps}}=\msf T_\ka(t)
\quad\text{and}\quad
\lim_{\eps\to0}\msf M_{\ka,\eps}(t)\mr e^{-t\msf c_{\ka,\eps}}=\msf M_\ka(t)
\qquad\text{in probability}.
\end{align}
Indeed, once the limits in \eqref{Equation: Moment Formula 1 Limits in Probability} are proved,
Proposition \ref{Proposition: Moments Prelimit} is a consequence of the
following technical lemma and the Vitali convergence theorem:

\begin{lemma}
\label{Lemma: Moments Prelimit UI}
For every $\ka>0$, there exists some $\th_\ka>0$ such that
for every $m=1,2$ and $t\in(0,\th_\ka)$ the following sequences of random variables are uniformly integrable:
\[\big\{\msf T_{\ka,\eps}(t)^m\mr e^{-mt\msf c_{\ka,\eps}}:\eps>0\big\}
\qquad\text{and}\qquad
\big\{\msf M_{\ka,\eps}(t)^m\mr e^{-mt\msf c_{\ka,\eps}}:\eps>0\big\}.\]
\end{lemma}

Lemma \ref{Lemma: Moments Prelimit UI} is proved in Section \ref{Section: Moments Prelimit UI} below
by combining the Feynman-Kac formula in \eqref{Equation: F-K} and the properties of intersection
local times stated in Section \ref{Section: SILT and MILT}.

We thus turn our focus to \eqref{Equation: Moment Formula 1 Limits in Probability}.
By \eqref{Equation: H ka eps Spectral Expansion}, for every $\ka,\eps,t>0$,
\[\msf T_{\ka,\eps}(t)\mr e^{-t\msf c_{\ka,\eps}}=\sum_{n=1}^\infty\mr e^{-t(\la_n(H_{\ka,\eps})+\msf c_{\ka,\eps})}\]
and
\[\msf M_{\ka,\eps}(t)\mr e^{-t\msf c_{\ka,\eps}}=\sum_{n=1}^\infty\mr e^{-t(\la_n(H_{\ka,\eps})+\msf c_{\ka,\eps})}\big\langle\psi_n(H_{\ka,\eps}),\mbf 1_D\big\rangle^2.\]
Thus, \eqref{Equation: Moment Formula 1 Limits in Probability} can be reformulated
into the claims that for every $\ka>0$ and $t\in(0,\th_\ka)$,
\begin{align}
\label{Equation: Moment Formula 1 Limits in Probability 2}
&\lim_{\eps\to0}\sum_{n=1}^\infty\mr e^{-t(\la_n(H_{\ka,\eps})+\msf c_{\ka,\eps})}=\sum_{n=1}^\infty\mr e^{-t\la_n(H_\ka)}\qquad\text{in probability},\\
\label{Equation: Moment Formula 1 Limits in Probability 3}
&\lim_{\eps\to0}\sum_{n=1}^\infty\mr e^{-t(\la_n(H_{\ka,\eps})+\msf c_{\ka,\eps})}\big\langle\psi_n(H_{\ka,\eps}),\mbf 1_D\big\rangle^2=\sum_{n=1}^\infty\mr e^{-t\la_n(H_\ka)}\big\langle\psi_n(H_\ka),\mbf 1_D\big\rangle^2\quad\text{in prob.}
\end{align}
Assumption \ref{Assumption: AH} states that for every fixed $n\geq1$,
one has
\[\mr e^{-t(\la_n(H_{\ka,\eps})+\msf c_{\ka,\eps})}\to\mr e^{-t\la_n(H_\ka)},
\quad\mr e^{-t(\la_n(H_{\ka,\eps})+\msf c_{\ka,\eps})}\big\langle\psi_n(H_{\ka,\eps}),\mbf 1_D\big\rangle^2\to\mr e^{-t\la_n(H_\ka)}\big\langle\psi_n(H),\mbf 1_D\big\rangle^2\]
as $\eps\to0$ in probability. The only difficulty in establishing \eqref{Equation: Moment Formula 1 Limits in Probability 2} and \eqref{Equation: Moment Formula 1 Limits in Probability 3}
is therefore to justify the limit when summing over all $n\geq1$. In order to get around this problem, we make use of the following technical result:

\begin{lemma}
\label{Lemma: Hilbert-Schmidt Cauchy Sequence}
Recall the definition of $\mc K_{\ka,\eps}(t)$ in \eqref{Equation: F-K}.
For any
$\ka>0$, there exists $\th_\ka>0$ such that for every $t\in(0,\th_\ka)$,
\[\lim_{\eps,\tilde\eps\to0}\mbf E\left[\big\|\mc K_{\ka,\eps}(t)\mr e^{-t\msf c_{\ka,\eps}}-\mc K_{\ka,\tilde\eps}(t)\mr e^{-t\msf c_{\ka,\tilde\eps}}\big\|_{L^2(D^2)}^2\right]=0.\]
\end{lemma}

This is proved in Section \ref{Section: Hilbert-Schmidt Cauchy Sequence}, also using the Feynman-Kac formula and the
properties of intersection local times.

Let $\ka>0$ and $t\in(0,\th_\ka)$ be fixed, where $\th_\ka$ is taken as in Lemma \ref{Lemma: Hilbert-Schmidt Cauchy Sequence}.
By combining Assumption \ref{Assumption: AH} with Lemma \ref{Lemma: Hilbert-Schmidt Cauchy Sequence},
every vanishing sequence of $\eps$'s has a subsequence $\eps_1>\eps_2>\cdots>0$ along which
the limits \eqref{Equation: Assumption Convergence of Eigenvalues}
and \eqref{Equation: Assumption Convergence of Eigenfunctions} hold almost surely for all $n\geq1$,
and, in addition,
for every $m,n\geq1$, one has
\begin{align}
\label{Equation: Hilbert-Schmidt Cauchy Sequence Subsequence}
\mbf E\left[\big\|\mc K_{\ka,\eps_m}(t)\mr e^{-t\msf c_{\ka,\eps_m}}-\mc K_{\ka,\eps_n}(t)\mr e^{-t\msf c_{\ka,\eps_n}}\big\|_{L^2(D^2)}^2\right]\leq8^{-\min\{m,n\}}.
\end{align}
Our aim is to show that the limits
\eqref{Equation: Moment Formula 1 Limits in Probability 2} and \eqref{Equation: Moment Formula 1 Limits in Probability 3}
hold almost surely along any such subsequence $\{\eps_n:n\geq1\}$, which will imply the desired convergence in probability.

Combining \eqref{Equation: Hilbert-Schmidt Cauchy Sequence Subsequence} with Markov's inequality, we get that
\[\mbf P\left[\big\|\mc K_{\ka,\eps_n}(t)\mr e^{-t\msf c_{\ka,\eps_n}}-\mc K_{\ka,\eps_{n+1}}(t)\mr e^{-t\msf c_{\ka,\eps_{n+1}}}\big\|_{L^2(D^2)}>2^{-n}\right]\leq\frac{8^{-n}}{4^{-n}}\leq 2^{-n}.\]
In particular, since $\sum_{n=1}^\infty2^{-n}<\infty$, it follows from the Borel-Cantelli lemma that
\[\mbf P\left[\sum_{n=1}^\infty\big\|\mc K_{\ka,\eps_n}(t)\mr e^{-t\msf c_{\ka,\eps_n}}-\mc K_{\ka,\eps_{n+1}}(t)\mr e^{-t\msf c_{\ka,\eps_{n+1}}}\big\|_{L^2(D^2)}<\infty\right]=1;\]
hence the sequence $\big\{\mc K_{\ka,\eps_n}(t):n\geq1\big\}$ is almost-surely Cauchy in $L^2(D^2)$.
Given that $L^2(D^2)$ is complete, we conclude that there exists some Hilbert-Schmidt integral operator $\mc K_\ka(t)\in L^2(D^2)$ such that
\begin{align}
\label{Equation: Hilbert-Schmidt Cauchy AS}
\lim_{n\to\infty}\big\|\mc K_{\ka,\eps_n}(t)\mr e^{-t\msf c_{\ka,\eps_n}}-\mc K_\ka(t)\big\|_{L^2(D^2)}^2=0\qquad\text{almost surely}.
\end{align}

Let us henceforth assume that we are working with a realization of $\xi$
in the probability-one event where \eqref{Equation: Assumption Convergence of Eigenvalues},
\eqref{Equation: Assumption Convergence of Eigenfunctions}, and \eqref{Equation: Hilbert-Schmidt Cauchy AS}
hold. We claim that, on this event,
\begin{align}
\label{Equation: K(t) Spectral Expansion}
\mc K_\ka(t)f=\mr e^{-t H_\ka}f=\sum_{n=1}^\infty\mr e^{-t\la_n(H_\ka)}\big\langle\psi_n(H_\ka),f\big\rangle\psi_n(H_\ka),
\qquad f\in L^2(D).
\end{align}
For this purpose: On the one hand, for every $m,n\geq1$, we have that
\begin{align*}
&\|\mc K_{\ka,\eps_m}(t)\mr e^{-t\msf c_{\ka,\eps_m}}\psi_n(H_{\ka,\eps_m})-\mc K_\ka(t)\psi_n(H_\ka)\|\\
&\leq\|\mc K_{\ka,\eps_m}(t)\mr e^{-t\msf c_{\ka,\eps_m}}\psi_n(H_{\ka,\eps_m})-\mc K_{\ka,\eps_m}(t)\mr e^{-t\msf c_{\ka,\eps_m}}\psi_n(H_\ka)\|\\
&\hspace{2in}+\|\mc K_{\ka,\eps_m}(t)\mr e^{-t\msf c_{\ka,\eps_m}}\psi_n(H_\ka)-\mc K_\ka(t)\psi_n(H_\ka)\|\\
&\leq\|\mc K_{\ka,\eps_m}(t)\mr e^{-t\msf c_{\ka,\eps_m}}\|_{L^2(D^2)}\|\psi_n(H_{\ka,\eps_m})-\psi_n(H_\ka)\|
+\|\mc K_{\ka,\eps_m}(t)\mr e^{-t\msf c_{\ka,\eps_m}}-\mc K_\ka(t)\|_{L^2(D^2)};
\end{align*}
\eqref{Equation: Assumption Convergence of Eigenfunctions} and \eqref{Equation: Hilbert-Schmidt Cauchy AS}
implies that this goes to zero as $m\to\infty$. On the other hand,
\eqref{Equation: H ka eps Spectral Expansion}, \eqref{Equation: Assumption Convergence of Eigenvalues}, and
\eqref{Equation: Assumption Convergence of Eigenfunctions} implies the $L^2$ limit
\[\mc K_{\ka,\eps_m}(t)\mr e^{-t\msf c_{\ka,\eps_m}}\psi_n(H_{\ka,\eps_m})=\mr e^{-t(\la_n(H_{\ka,\eps_m})+\msf c_{\ka,\eps_m})}\psi_n(H_{\ka,\eps_m})\to\mr e^{-t\la_n(H_\ka)}\psi_n(H_\ka)\]
as $m\to\infty$.
If we combine these two results, then we get that
\[\mc K_\ka(t)\psi_n(H_\ka)=\mr e^{-t\la_n(H_\ka)}\psi_n(H_\ka)\]
for all $n\geq1$; hence \eqref{Equation: K(t) Spectral Expansion} holds since $\psi_n(H_\ka)$ is an orthonormal basis.

With \eqref{Equation: K(t) Spectral Expansion} in hand, we are now in a position
to prove the almost-sure versions of
\eqref{Equation: Moment Formula 1 Limits in Probability 2} and \eqref{Equation: Moment Formula 1 Limits in Probability 3}
along the subsequence $\eps_n$, and thus conclude the proof of Proposition \ref{Proposition: Moments Prelimit}:
On the one hand, \eqref{Equation: H ka eps Spectral Expansion},
\eqref{Equation: Hilbert-Schmidt Cauchy AS} and \eqref{Equation: K(t) Spectral Expansion} imply that
\begin{multline*}
\lim_{m\to\infty}\sum_{n=1}^\infty\mr e^{-t(\la_n(H_{\ka,\eps_m})+\msf c_{\ka,\eps_m})}
=\lim_{m\to\infty}\|\mc K_{\ka,\eps_m}(t/2)\mr e^{-(t/2)\msf c_{\ka,\eps_m}}\|_{L^2(D^2)}^2\\
=\|\mc K_{\ka}(t/2)\|_{L^2(D^2)}^2
=\sum_{n=1}^\infty\mr e^{-t\la_n(H_\ka)},
\end{multline*}
which yields \eqref{Equation: Moment Formula 1 Limits in Probability 2}.
On the other hand, the same three results and the Cauchy-Schwarz
inequality imply that
\begin{align*}
&\lim_{m\to\infty}\left|\sum_{n=1}^\infty\mr e^{-t(\la_n(H_{\ka,\eps_m})+\msf c_{\ka,\eps_m})}\big\langle\psi_n(H_{\ka,\eps_m}),\mbf 1_D\big\rangle^2-\sum_{n=1}^\infty\mr e^{-t\la_n(H_\ka)}\big\langle\psi_n(H_\ka),\mbf 1_D\big\rangle^2\right|\\
&=\lim_{m\to\infty}\left|\big\langle\mc K_{\ka,\eps_m}(t)\mbf 1_D-\mc K_\ka(t)\mbf 1_D,\mbf 1_D\big\rangle\right|\\
&\leq\lim_{m\to\infty}\|\mc K_{\ka,\eps_m}(t)\mbf 1_D-\mc K_\ka(t)\mbf 1_D\|\|\mbf 1_D\|\\
&\leq\lim_{m\to\infty}\|\mc K_{\ka,\eps_m}(t)-\mc K_\ka(t)\|_{L^2(D^2)}\|\mbf 1_D\|^2=0,
\end{align*}
which yields \eqref{Equation: Moment Formula 1 Limits in Probability 3}.

\subsection{Proof of Lemma \ref{Lemma: Moments Prelimit UI}}
\label{Section: Moments Prelimit UI}

By \eqref{Equation: Poussin},
it suffices to show that for every integer $m\geq1$
(for the purposes of Lemma \ref{Lemma: Moments Prelimit UI}, $m=3$ suffices), there exists some $\th_{\ka,m}>0$ such that for all $t\in(0,\th_{\ka,m})$, one has
\begin{align}
\label{Equation: DVP Moments Prelimit UI}
\sup_{\eps>0}\mbf E\left[\msf T_{\ka,\eps}(t)^m\mr e^{-mt\msf c_{\ka,\eps}}\right]<\infty
\qquad\text{and}\qquad
\sup_{\eps>0}\mbf E\left[\msf M_{\ka,\eps}(t)^m\mr e^{-mt\msf c_{\ka,\eps}}\right]<\infty.
\end{align}

Our first step in proving \eqref{Equation: DVP Moments Prelimit UI} is to provide Feynman-Kac
formulas for the moments appearing therein---this will be used repeatedly in future results in the paper
as well. Toward this end,
by combining \eqref{Equation: F-K} and \eqref{Equation: H ka eps Spectral Expansion}, we get that
\begin{multline}
\label{Equation: UnAveraged F-K Bridge}
\msf T_{\ka,\eps}(t)\mr e^{-t\msf c_{\ka,\eps}}
=\mr{Tr}\big[\mr e^{-tH_{\ka,\eps}}\big]
=\int_D\mc K_{\ka,\eps}(t;x,x)\d x\\
=\frac{1}{2\pi t}\int_D\mbf E_B\left[\mbf 1_{\{\tau_D(B^{x,x}_t)>t\}}\exp\left(-\ka\int_0^t\xi_\eps\big(B^{x,x}_t(r)\big)\d r-t\msf c_{\ka,\eps}\right)\right]\d x,
\end{multline}
and similarly
\begin{multline}
\label{Equation: UnAveraged F-K Motion}
\msf M_{\ka,\eps}(t)\mr e^{-t\msf c_{\ka,\eps}}
=\int_D\mr e^{-tH_{\ka,\eps}}\mbf 1_D(x)\d x
=\int_{D^2}\mc K_{\ka,\eps}(t;x,y)\d x\dd y\\
=\int_D\mbf E_B\left[\mbf 1_{\{\tau_D(B^x)>t\}}\exp\left(-\ka\int_0^t\xi_\eps\big(B^x(r)\big)\d r-t\msf c_{\ka,\eps}\right)\right]\d x.
\end{multline}
If we apply Tonelli's theorem to \eqref{Equation: UnAveraged F-K Bridge}, then for any integer power $m\geq1$, we have
\begin{multline}
\label{Equation: UnAveraged F-K Bridge 2}
\msf T_{\ka,\eps}(t)^m\mr e^{-mt\msf c_{\ka,\eps}}\\
=\frac{1}{(2\pi t)^m}\int_{D^m}\prod_{i=1}^m\mbf E_B\left[\mbf 1_{\{\tau_D(B^{x_i,x_i}_t)>t\}}\exp\left(-\ka\int_0^t\xi_\eps\big(B^{x_i,x_i}_t(r)\big)\d r-t\msf c_{\ka,\eps}\right)\right]\d x.
\end{multline}
Note that for any functional $F$,
\begin{align}
\label{Equation: Product Means Independent}
\prod_{i=1}^m\mbf E\big[F(B^{x_i,x_i}_t)\big]=\mbf E\left[\prod_{i=1}^mF({_i}B^{x_i,x_i}_t)\right],
\end{align}
where we recall our notation for independent Brownian motions/bridges in Definition \ref{Definition: Brownian}.
Therefore, \eqref{Equation: UnAveraged F-K Bridge 2} can be written as
\begin{multline}
\label{Equation: UnAveraged F-K Bridge 3}
\msf T_{\ka,\eps}(t)^m\mr e^{-mt\msf c_{\ka,\eps}}\\
=\frac{1}{(2\pi t)^m}\int_{D^m}\mbf E_B\left[\mbf 1_{\cap_{i\leq m}\{\tau_D({_i}B^{x_i,x_i}_t)>t\}}\exp\left(\sum_{i=1}^m\left\{-\ka\int_0^t\xi_\eps\big({_i}B^{x_i,x_i}_t(r)\big)\d r-t\msf c_{\ka,\eps}\right\}\right)\right]\d x,
\end{multline}
where we assume that, in addition of being independent of each other, the ${_i}B^{x_i,x_i}_t$'s are independent of $\xi$.
If we apply Tonelli's theorem to \eqref{Equation: UnAveraged F-K Bridge 3}, we then get
\begin{multline}
\label{Equation: UnAveraged F-K Bridge 4}
\mbf E\big[\msf T_{\ka,\eps}(t)^m\mr e^{-mt\msf c_{\ka,\eps}}\big]\\
=\frac{1}{(2\pi t)^m}\int_{D^m}\mbf E_B\left[\mbf 1_{\cap_{i\leq m}\{\tau_D({_i}B^{x_i,x_i}_t)>t\}}\mbf E_{\xi}\left[\exp\left(\sum_{i=1}^m\left\{-\ka\int_0^t\xi_\eps\big({_i}B^{x_i,x_i}_t(r)\big)\d r-t\msf c_{\ka,\eps}\right\}\right)\right]\right]\d x,
\end{multline}
where $\mbf E_{\xi}$ denotes the expectation with respect to $\xi$ only, conditional on the ${_i}B^{x_i,x_i}_t$'s.
Since $\xi_\eps$ is a Gaussian process, conditional on the ${_i}B^{x_i,x_i}_t$'s, the random variable
\[\sum_{i=1}^m\left\{-\ka\int_0^t\xi_\eps\big({_i}B^{x_i,x_i}_t(r)\big)\d r-t\msf c_{\ka,\eps}\right\}\]
is Gaussian with mean $-mt\msf c_{\ka,\eps}$ and variance
\begin{align}
\label{Equation: UnAveraged F-K Bridge 5}
\ka^2\sum_{i,j=1}^m\int_{[0,t]^2}\mbf E\Big[\xi_\eps\big({_i}B^{x_i,x_i}_t(r_1)\big)\,\xi_\eps\big({_j}B^{x_j,x_j}_t(r_2)\big)\Big]\d r.
\end{align}
At this point, if we recall that $\xi_\eps$'s covariance is given by the
heat semigroup $p_\eps$, as well as the definitions of approximate SILTs and MILTs
in Definition \ref{Definition: Approximate SILT and MILT} (and the shorthands in
Notations \ref{Notation: SILT Shorthand} and
\ref{Notation: MILT Shorthand}), then we obtain that
\[\eqref{Equation: UnAveraged F-K Bridge 5}=2\ka^2\sum_{i=1}^m\be^\eps_t\big({_i}B^{x_i,x_i}_t\big)
+2\ka^2\sum_{1\leq i<j\leq m}\al^\eps_t\big({_i}B^{x_i,x_i}_t,{_j}B^{x_j,x_j}_t\big),\]
where the factor of $2$ in front of SILTs comes form $\be^\eps_{[0,t]}(Z)=2\be^\eps_{[0,t]^2_\leq}(Z)=2\be^\eps_t(Z)$,
and the factor of $2$ in front of MILTs comes from the fact that $\al^\eps_t(Z_i,Z_j)=\al^\eps_t(Z_j,Z_i)$.
If we now use this information to perform a Gaussian moment generating function calculation in \eqref{Equation: UnAveraged F-K Bridge 4},
then we are finally led to our first moment formula:
\begin{multline}
\label{Equation: Averaged F-K Bridge}
\mbf E\big[\msf T_{\ka,\eps}(t)^m\mr e^{-mt\msf c_{\ka,\eps}}\big]
=\frac{1}{(2\pi t)^m}\int_{D^m}\mbf E\Bigg[\mbf 1_{\cap_{i\leq m}\{\tau_D({_i}B^{x_i,x_i}_t)>t\}}\\
\cdot\exp\left(\sum_{i=1}^m\Big(\ka^2\be^\eps_t\big({_i}B^{x_i,x_i}_t\big)-t\msf c_{\ka,\eps}\Big)
+\ka^2\sum_{1\leq i<j\leq m}\al^\eps_t\big({_i}B^{x_i,x_i}_t,{_j}B^{x_j,x_j}_t\big)\right)\Bigg]\d x.
\end{multline}
If we perform essentially the same argument that took us from \eqref{Equation: UnAveraged F-K Bridge 2}
to \eqref{Equation: Averaged F-K Bridge}, but replace the Brownian bridges by Brownian motions, then
we get from \eqref{Equation: UnAveraged F-K Motion} that
\begin{multline}
\label{Equation: Averaged F-K Motion}
\mbf E\big[\msf M_{\ka,\eps}(t)^m\mr e^{-mt\msf c_{\ka,\eps}}\big]
=\int_{D^m}\mbf E\Bigg[\mbf 1_{\cap_{i\leq m}\{\tau_D({_i}B^{x_i})>t\}}\\
\cdot\exp\left(\sum_{i=1}^m\Big(\ka^2\be^\eps_t\big({_i}B^{x_i}\big)-t\msf c_{\ka,\eps}\Big)
+\ka^2\sum_{1\leq i<j\leq m}\al^\eps_t\big({_i}B^{x_i},{_j}B^{x_j}\big)\right)\Bigg]\d x.
\end{multline}

We now return to the task of proving \eqref{Equation: DVP Moments Prelimit UI}. For this purpose,
if we use the trivial bound $\mbf 1_{\{\cdot\}}\leq1$ and H\"older's inequality in
\eqref{Equation: Averaged F-K Bridge} and \eqref{Equation: Averaged F-K Motion},
together with the fact that $D$ is bounded,
then we get that \eqref{Equation: DVP Moments Prelimit UI} is a consequence of the following claim:
For any $\eta\geq1$, if $t>0$ is small enough, then
\begin{multline}
\label{Equation: UI after Holder}
\sup_{\eps>0,~x\in D^m}\prod_{i=1}^m\mbf E\left[\exp\Big(\eta c_m\big(\ka^2\be^\eps_t(Z_i)-t\msf c_{\ka,\eps}\big)\Big)\right]^{1/c_m}\\
\cdot\prod_{1\leq i<j\leq m}\mbf E\left[\exp\Big(\eta c_m\ka^2\al^\eps_t(Z_i,Z_j)\Big)\right]^{1/c_m}<\infty,
\end{multline}
where $c_m=\frac{m(m+1)}{2}$, and $Z_i$ is either ${_i}B^{x_i}$ or ${_i}B^{x_i,x_i}_t$ for $1\leq i\leq m$.
(For the purpose of this proof, we only need this result for $\eta=1$, but we will need
$\eta>1$ later on.)
The finiteness of the suprema involving $\al^\eps_t\big(Z_i,Z_j\big)$ for $i\neq j$
and small enough $t$
follows immediately from Lemma \ref{Lemma: MILT is UI}.
As for $\be^\eps_t(Z_i)-t\msf c_{\ka,\eps}$, if we write
\begin{align}
\label{Equation: Renormalization/Expected Value Rearrangement}
\ka^2\be^\eps_t(Z)-t\msf c_{\ka,\eps}=\ka^2\Big(\be^\eps_t(Z)-\mbf E\big[\be^\eps_t(Z)\big]\Big)+\ka^2\mbf E\big[\be^\eps_t(Z)\big]-t\msf c_{\ka,\eps},
\end{align}
then the finiteness of the suprema follows from Lemmas \ref{Lemma: Uniform Integrability of SILT}
and \ref{Lemma: Smoothed Expectation of SILT}-(1) (the latter of which
implies that $\ka^2\mbf E\big[\be^\eps_t(Z)\big]-t\msf c_{\ka,\eps}$ converges to
a constant as $\eps\to0$ thanks to \eqref{Equation: Renormalization}).
This concludes the proof of Lemma \ref{Lemma: Moments Prelimit UI}.

\subsection{Proof of Proposition \ref{Proposition: Feynman-Kac for T and M}}
\label{Section: Feynman-Kac for T and M}

We prove the following result, which obviously implies Proposition \ref{Proposition: Feynman-Kac for T and M}
by \eqref{Equation: Product Means Independent} and the fact that $\mbf{Var}[X]=\mbf E[X^2]-\mbf E[X]^2$: For any integer $m\geq1$,
there exists some $\th_{\ka,m}>0$ such that for every $t\in(0,\th_{\ka,m})$, one has
\begin{multline}
\label{Equation: Averaged F-K Bridge Limit}
\lim_{\eps\to0}\mbf E\big[\msf T_{\ka,\eps}(t)^m\mr e^{-mt\msf c_{\ka,\eps}}\big]
=\frac{\mr e^{m\ka^2t\log t/2\pi}}{(2\pi t)^m}\int_{D^m}\mbf E\Bigg[\mbf 1_{\cap_{i\leq m}\{\tau_D({_i}B^{x_i,x_i}_t)>t\}}\\
\cdot\exp\left(\ka^2\sum_{i=1}^m\ga_t\big({_i}B^{x_i,x_i}_t\big)
+\ka^2\sum_{1\leq i<j\leq m}\al_t\big({_i}B^{x_i,x_i}_t,{_j}B^{x_j,x_j}_t\big)\right)\Bigg]\d x,
\end{multline}
and similarly,
\begin{multline}
\label{Equation: Averaged F-K Motion Limit}
\lim_{\eps\to0}\mbf E\big[\msf M_{\ka,\eps}(t)^m\mr e^{-mt\msf c_{\ka,\eps}}\big]
=\mr e^{m\ka^2(t\log t-t)/2\pi}\int_{D^m}\mbf E\Bigg[\mbf 1_{\cap_{i\leq m}\{\tau_D({_i}B^{x_i})>t\}}\\
\cdot\exp\left(\ka^2\sum_{i=1}^m\ga_t\big({_i}B^{x_i}\big)
+\ka^2\sum_{1\leq i<j\leq m}\al_t\big({_i}B^{x_i},{_j}B^{x_j}\big)\right)\Bigg]\d x.
\end{multline}
These two limits follow by combining the moment formulas in 
\eqref{Equation: Averaged F-K Bridge} and \eqref{Equation: Averaged F-K Motion}
with the following observations:
\begin{enumerate}[(1)]
\item If we let $Z_i$ be either ${_i}B^{x_i}$ or ${_i}B^{x_i,x_i}_t$ for $1\leq i\leq m$,
then by Theorem \ref{Theorem: Existence of MILT},
\[\lim_{\eps\to0}\sum_{1\leq i<j\leq m}\al^\eps_t(Z_i,Z_j)
=\sum_{1\leq i<j\leq m}\al_t(Z_i,Z_j)
\quad\text{in probability.}\]
\item By Theorem \ref{Theorem: Existence of SILT}-(2),
the rearrangement \eqref{Equation: Renormalization/Expected Value Rearrangement},
and the definition of $\msf c_{\ka,\eps}$ in \eqref{Equation: Renormalization} and Lemma \ref{Lemma: Smoothed Expectation of SILT}-(1)
(noting that $\int_0^t\log r-\log(t-r)\d r=0$),
\begin{align*}
\lim_{\eps\to0}\sum_{i=1}^m\Big(\ka^2\be^\eps_t\big({_i}B^{x_i,x_i}_t\big)-t\msf c_{\ka,\eps}\Big)
=\ka^2\sum_{i=1}^m\ga_t\big({_i}B^{x_i,x_i}_t\big)+\frac{m\ka^2t\log t}{2\pi}\quad\text{in probability};\\
\lim_{\eps\to0}\sum_{i=1}^m\Big(\ka^2\be^\eps_t\big({_i}B^{x_i}\big)-t\msf c_{\ka,\eps}\Big)
=\ka^2\sum_{i=1}^m\ga_t\big({_i}B^{x_i}\big)+\frac{m\ka^2(t\log t-t)}{2\pi}\quad\text{in probability}.
\end{align*}
\item By \eqref{Equation: UI after Holder} with any $\eta>1$, we get that, for small enough $t$,
\begin{enumerate}[(3.1)]
\item the random variables inside the expectations in \eqref{Equation: Averaged F-K Bridge Limit}
and \eqref{Equation: Averaged F-K Motion Limit} are uniformly integrable over $\eps>0$ for every fixed $x_i$
by \eqref{Equation: Poussin}; and
\item the expected values inside the $\dd x$ integrals  in \eqref{Equation: Averaged F-K Bridge Limit}
and \eqref{Equation: Averaged F-K Motion Limit} (which we view as functions of the $x_i$'s)
are uniformly bounded over $\eps>0$.
\end{enumerate}
Thus, we can bring the $\eps\to0$ limits in \eqref{Equation: Averaged F-K Bridge Limit}
and \eqref{Equation: Averaged F-K Motion Limit} inside the $\dd x$ integrals and the expectations
by the dominated convergence theorem and the Vitali convergence theorem.
\end{enumerate}

\begin{remark}
\label{Remark: Renormalization 3}
Following-up on Remarks \ref{Remark: Renormalization 1} and \ref{Remark: Renormalization 2},
by Lemma \ref{Lemma: Smoothed Expectation of SILT}, we see from observation (2) above
that \eqref{Equation: Alternate Renormalization} is the only type of renormalization which will
yield a nontrivial limit for the moments of $\msf T_{\ka,\eps}(t)^m\mr e^{-mt\msf c_{\ka,\eps}}$
and $\msf M_{\ka,\eps}(t)^m\mr e^{-mt\msf c_{\ka,\eps}}$. Moreover,
the appearance of logarithmic terms in the moments of $\msf T_\ka(t)$ and $\msf M_\ka(t)$
comes from the expectations of approximate SILTs.
\end{remark}

\subsection{Proof of Lemma \ref{Lemma: Hilbert-Schmidt Cauchy Sequence}}
\label{Section: Hilbert-Schmidt Cauchy Sequence}

We first note that for any fixed $\eps,\tilde\eps>0$ we can expand
\begin{multline}
\label{Equation: HS Expansion}
\left\lVert
	\mc K_{\kappa, \eps}(t)\mr e^{-t\msf c_{\ka,\eps}}-\mc K_{\kappa, \tilde\eps}(t)\mr e^{-t\msf c_{\ka,\tilde\eps}}
	\right\rVert_{L^2(D^2)}^2\\
	=\int_{D^2}\left(
	\mc K_{\kappa, \eps}(t;x,y)\mr e^{-t\msf c_{\ka,\eps}}-\mc K_{\kappa, \tilde\eps}(t;x,y)\mr e^{-t\msf c_{\ka,\tilde\eps}}
	\right)^2\d x\dd y
\end{multline}
as the following sum:
\begin{multline}
\label{Equation: HS Expansion 2}
\int_{D^2}
	\mc K_{\kappa,\eps}(t;x,y)\mc K_{\kappa, \eps}(t;x,y)\mr e^{-2t\msf c_{\ka,\eps}}+\mc K_{\kappa,\tilde\eps}(t;x,y)\mc K_{\kappa, \tilde\eps}(t;x,y)\mr e^{-2t\msf c_{\ka,\tilde\eps}}\\
	-2\mc K_{\kappa, \eps}(t;x,y)\mr e^{-t\msf c_{\ka,\eps}}\cdot\mc K_{\kappa, \tilde\eps}(t;x,y)\mr e^{-t\msf c_{\ka,\tilde\eps}}
	\d x\dd y.
\end{multline}
Given that $\mc K_{\ka,\eps}(t)$ is a symmetric semigroup for all $\ka,\eps>0$,
\begin{align}
\label{Equation: Symmetry and Semigroup Integral Along the Middle}
\int_{D}\mc K_{\ka,\eps}(t;x,y)\mc K_{\ka,\eps}(t;x,y)\d y
=\int_{D}\mc K_{\ka,\eps}(t;x,y)\mc K_{\ka,\eps}(t;y,x)\d y
=\mc K_{\ka,\eps}(2t;x,x)
\end{align}
for all $x\in D$. If we apply this to the first
line in \eqref{Equation: HS Expansion 2}, then we can write \eqref{Equation: HS Expansion} as
\begin{multline}
\label{Equation: HS Expansion 3}
\int_D\mc K_{\kappa,\eps}(2t;x,x)\mr e^{-2t\msf c_{\ka,\eps}}+\mc K_{\kappa,\tilde\eps}(2t;x,x)\mr e^{-2t\msf c_{\ka,\tilde\eps}}\d x\\
	-2\int_{D^2}
\mc K_{\kappa, \eps}(t;x,y)\mr e^{-t\msf c_{\ka,\eps}}\cdot\mc K_{\kappa, \tilde\eps}(t;x,y)\mr e^{-t\msf c_{\ka,\tilde\eps}}
	\d x\dd y.
\end{multline}
Note that \eqref{Equation: H ka eps Spectral Expansion} and \eqref{Equation: Averaged F-K Bridge Limit} imply that
for small enough $t>0$,
\begin{multline*}
\lim_{\eps\to0}\mbf E\left[\int_D\mc K_{\kappa,\eps}(2t;x,x)\d x\,\,\mr e^{-2t\msf c_{\ka,\eps}}\right]
=\lim_{\eps\to0}\mbf E\big[\msf T_{\ka,\eps}(2t)\mr e^{-2t\msf c_{\ka,\eps}}\big]\\
=\frac{\mr e^{\ka^2t\log(2t)/\pi}}{4\pi t}
\int_{D}\mbf E\Bigg[\mbf 1_{\{\tau_D(B^{x,x}_{2t})>2t\}}
\mr e^{\ka^2\ga_{2t}(B^{x,x}_{2t})}\Bigg]\d x.
\end{multline*}
If we plug this into the first line of \eqref{Equation: HS Expansion 3},
then we conclude from the second line of \eqref{Equation: HS Expansion 3}
that Lemma \ref{Lemma: Hilbert-Schmidt Cauchy Sequence} follows if we show that for small enough $t>0$, one has
\begin{multline}
\label{Equation: HS Expansion GOAL}
\lim_{\eps,\tilde\eps\to0}
\mbf E\left[\int_{D^2}
\mc K_{\kappa, \eps}(t;x,y)\mr e^{-t\msf c_{\ka,\eps}}\cdot\mc K_{\kappa, \tilde\eps}(t;x,y)\mr e^{-t\msf c_{\ka,\tilde\eps}}
	\d x\dd y\right]\\
	=\frac{\mr e^{\ka^2t\log(2t)/\pi}}{4\pi t}
\int_{D}\mbf E\Bigg[\mbf 1_{\{\tau_D(B^{x,x}_{2t})>2t\}}
\mr e^{\ka^2\ga_{2t}(B^{x,x}_{2t})}\Bigg]\d x.
\end{multline}

Toward this end, we begin by deriving the analogue of \eqref{Equation: Symmetry and Semigroup Integral Along the Middle}
in the case where $\eps\neq\tilde\eps$, which will be instrumental in getting a formula for the right-hand side of
\eqref{Equation: HS Expansion GOAL}. For this purpose, we use the Feynman-Kac formula
\eqref{Equation: F-K} to write the integral
\begin{align}
\label{Equation: HS Expansion GOAL 1}
\int_{D}
\mc K_{\kappa, \eps}(t;x,y)\mc K_{\kappa, \tilde\eps}(t;x,y)\d y
\end{align}
as the probabilistic expression
\begin{multline}
\label{Equation: HS Expansion GOAL 2}
\int_Dp_t(x-y)
\mbf E_B\left[\mbf 1_{\{\tau_D(B^{x,y}_t)>t\}}\exp\left(-\ka\int_0^t\xi_\eps\big(B^{x,y}_t(r)\big)\d r\right)\right]\\
\cdot p_t(x-y)\mbf E_B\left[\mbf 1_{\{\tau_D(B^{x,y}_t)>t\}}\exp\left(-\ka\int_0^t\xi_{\tilde\eps}\big(B^{x,y}_t(r)\big)\d r\right)\right]\d y.
\end{multline}
If we use the fact that $p_t(\cdot)$ is an even function together with a time reversal on the Brownian bridge $B^{x,y}_t$ in
the expectation on the second line of \eqref{Equation: HS Expansion GOAL 2},
and then combine the product of the two expectations in \eqref{Equation: HS Expansion GOAL 2}
into a single expectation with two independent Brownian bridges, then we get that \eqref{Equation: HS Expansion GOAL 1} is
equal to
\begin{multline}
\label{Equation: HS Expansion GOAL 3}
\int_Dp_t(x-y)p_t(y-x)
\mbf E_B\Bigg[\mbf 1_{\{\tau_D({_1}B^{x,y}_t)>t\}\cap\{\tau_D({_2}B^{y,x}_t)>t\}}\\
\cdot\exp\left(-\ka\int_0^t\xi_{\eps}\big({_1}B^{x,y}_t(r)\big)\d r-\ka\int_0^t\xi_{\tilde\eps}\big({_2}B^{y,x}_t(r)\big)\d r\right)\Bigg]\d y.
\end{multline}
Consider the Brownian bridge $B^{x,x}_{2t}$. If we condition on the midpoint $B^{x,x}_{2t}(t)$
being equal to $y$, then the path segments
\[\big(B^{x,x}_{2t}(r):r\in[0,t]\big)
\qquad\text{and}\qquad
\big(B^{x,x}_{2t}(t+r):r\in[0,t]\big)\]
before and after the midpoint are equal in joint distribution to
\[\big({_1}B^{x,y}_t(r):r\in[0,t]\big)
\qquad\text{and}\qquad
\big({_2}B^{y,x}_t(r):r\in[0,t]\big).\]
Given that $B^{x,x}_{2t}(t)$'s density function is given by
\[y\mapsto\frac{p_t(x-y)p_t(y-x)}{p_{2t}(x,x)},\qquad y\in\mbb R^2,\]
if we multiply and divide \eqref{Equation: HS Expansion GOAL 3} by $p_{2t}(x,x)$
and then carry out the integration with respect to $\dd y$, then we get from the law of total expectation
that \eqref{Equation: HS Expansion GOAL 1} is equal to
\begin{multline}
\label{Equation: HS Expansion GOAL 4}
p_{2t}(x,x)
\mbf E_B\Bigg[\mbf 1_{\{\tau_D(B^{x,x}_{2t})>2t\}}\exp\left(-\ka\int_0^t\xi_{\eps}\big(B^{x,x}_{2t}(r)\big)\d r-\ka\int_t^{2t}\xi_{\tilde\eps}\big(B^{x,x}_{2t}(r)\big)\d r\right)\Bigg]\d y.
\end{multline}

With this in hand, we are now in a position to calculate the expectation on the right-hand side of \eqref{Equation: HS Expansion GOAL}:
Conditional on $B^{x,x}_{2t}$'s path, the random variable inside the expectation and the exponential in \eqref{Equation: HS Expansion GOAL 4}
is Gaussian with mean zero and variance
\[\mbf E_\xi\left[\left(-\ka\int_0^t\xi_{\eps}\big(B^{x,x}_{2t}(r)\big)\d r-\ka\int_t^{2t}\xi_{\tilde\eps}\big(B^{x,x}_{2t}(r)\big)\d r\right)^2\right],\]
which we can expand into the sum of four terms:
\begin{align}
\label{Equation: HS Expansion GOAL COV 1}
\ka^2\int_{[0,t]^2}\mbf E_\xi\Big[\xi_{\eps}\big(B^{x,x}_{2t}(r_1)\big)\xi_{\eps}\big(B^{x,x}_{2t}(r_2)\big)\Big]\d r;
\end{align}
\begin{align}
\label{Equation: HS Expansion GOAL COV 2}
\ka^2\int_{[t,2t]^2}\mbf E_\xi\Big[\xi_{\tilde\eps}\big(B^{x,x}_{2t}(r_1)\big)\xi_{\tilde\eps}\big(B^{x,x}_{2t}(r_2)\big)\Big]\d r;
\end{align}
\begin{align}
\label{Equation: HS Expansion GOAL COV 3}
\ka^2\int_{[0,t]\times[t,2t]}\mbf E_\xi\Big[\xi_{\eps}\big(B^{x,x}_{2t}(r_1)\big)\xi_{\tilde\eps}\big(B^{x,x}_{2t}(r_2)\big)\Big]\d r;
\end{align}
\begin{align}
\label{Equation: HS Expansion GOAL COV 4}
\ka^2\int_{[t,2t]\times[0,t]}\mbf E_\xi\Big[\xi_{\tilde\eps}\big(B^{x,x}_{2t}(r_1)\big)\xi_{\eps}\big(B^{x,x}_{2t}(r_2)\big)\Big]\d r.
\end{align}
By replicating the calculation we performed earlier in \eqref{Equation: UnAveraged F-K Bridge 5},
we get that
\[\eqref{Equation: HS Expansion GOAL COV 1}+\eqref{Equation: HS Expansion GOAL COV 2}=2\ka^2\be^\eps_{[0,t]^2_{\leq}}(B^{x,x}_{2t})+2\ka^2\be^{\tilde\eps}_{[t,2t]^2_{\leq}}(B^{x,x}_{2t}).\]
If we perform a similar calculation, but this time use the covariance property
\eqref{Equation: General xi eps Covariance}, then we get that
\[\eqref{Equation: HS Expansion GOAL COV 3}+\eqref{Equation: HS Expansion GOAL COV 4}
=2\ka^2\be^{(\eps+\tilde\eps)/2}_{[0,t]\times[t,2t]}(B^{x,x}_{2t}).\]
In particular, a Gaussian moment generating function calculation in \eqref{Equation: HS Expansion GOAL 4} implies that \eqref{Equation: HS Expansion GOAL}
can be reformulated as the claim that
\begin{multline}
\label{Equation: HS Expansion GOAL Averaged}
\lim_{\eps,\tilde\eps\to0}
\frac{1}{4\pi t}
\int_{D}\mbf E\Bigg[\mbf 1_{\{\tau_D(B^{x,x}_{2t})>2t\}}
\mr e^{(\ka^2\be^\eps_{[0,t]^2_{\leq}}(B^{x,x}_{2t})-t\msf c_{\ka,\eps})+(\ka^2\be^{\tilde\eps}_{[t,2t]^2_{\leq}}(B^{x,x}_{2t})-t\msf c_{\ka,\tilde\eps})}\\
\mr e^{\ka^2\be^{(\eps+\tilde\eps)/2}_{[0,t]\times[t,2t]}(B^{x,x}_{2t})}\Bigg]\d x
	=\frac{\mr e^{\ka^2t\log(2t)/\pi}}{4\pi t}
\int_{D}\mbf E\Bigg[\mbf 1_{\{\tau_D(B^{x,x}_{2t})>2t\}}
\mr e^{\ka^2\ga_{2t}(B^{x,x}_{2t})}\Bigg]\d x.
\end{multline}

Toward this end, we first note that we can write
\[\big(\ka^2\be^\eps_{[0,t]^2_{\leq}}(B^{x,x}_{2t})-t\msf c_{\ka,\eps}\big)+\big(\ka^2\be^{\tilde\eps}_{[t,2t]^2_{\leq}}(B^{x,x}_{2t})-t\msf c_{\ka,\tilde\eps}\big)
+\ka^2\be^{(\eps+\tilde\eps)/2}_{[0,t]\times[t,2t]}(B^{x,x}_{2t})\]
(i.e., the random variable inside the exponential on the left-hand side of \eqref{Equation: HS Expansion GOAL Averaged})
as the sum of the following two terms:
\begin{multline}
\label{Equation: HS Expansion GOAL Averaged 1}
\Big(\ka^2\be^\eps_{[0,t]^2_{\leq}}(B^{x,x}_{2t})-\ka^2\mbf E\big[\be^\eps_{[0,t]^2_{\leq}}(B^{x,x}_{2t})\big]\Big)
+\Big(\ka^2\be^{\tilde\eps}_{[t,2t]^2_{\leq}}(B^{x,x}_{2t})-\ka^2\mbf E\big[\be^{\tilde\eps}_{[t,2t]^2_{\leq}}(B^{x,x}_{2t})\big]\Big)\\
+\Big(\ka^2\be^{(\eps+\tilde\eps)/2}_{[0,t]\times[t,2t]}(B^{x,x}_{2t})-\ka^2\mbf E\big[\be^{(\eps+\tilde\eps)/2}_{[0,t]\times[t,2t]}(B^{x,x}_{2t})\big]\Big),
\end{multline}
and
\begin{multline}
\label{Equation: HS Expansion GOAL Averaged 2}
\Big(\ka^2\mbf E\big[\be^\eps_{[0,t]^2_{\leq}}(B^{x,x}_{2t})\big]-t\msf c_{\ka,\eps}\Big)
+\Big(\ka^2\mbf E\big[\be^{\tilde\eps}_{[t,2t]^2_{\leq}}(B^{x,x}_{2t})\big]-t\msf c_{\ka,\tilde\eps}\Big)\\
+\ka^2\mbf E\big[\be^{(\eps+\tilde\eps)/2}_{[0,t]\times[t,2t]}(B^{x,x}_{2t})\big].
\end{multline}
The reason why we write the terms in this particular way is that it clearly emphasizes the following
two facts: On the one hand, by Theorem \ref{Theorem: Existence of SILT},
\begin{align}
\label{Equation: HS Expansion GOAL Averaged 1.1}
\lim_{\eps,\tilde\eps\to0}\eqref{Equation: HS Expansion GOAL Averaged 1}=\ga_{[0,t]^2_\leq}(B^{x,x}_{2t})+\ga_{[t,2t]^2_\leq}(B^{x,x}_{2t})
+\be_{[0,t]\times[t,2t]}(B^{x,x}_{2t})-\mbf E\big[\be_{[0,t]\times[t,2t]}(B^{x,x}_{2t})\big]
\end{align}
in probability.
Given that $A\mapsto\be^\eps_A(Z)$ is a measure on $[0,\infty)^2$
that is zero whenever $A$ has Lebesgue measure zero
(this is clear from Definition \ref{Definition: Approximate SILT and MILT}),
it follows that $\be^\eps_A(Z)+\be^\eps_{\tilde A}(Z)=\be^\eps_{A\cup\tilde A}(Z)$ whenever $A\cap\tilde A$ has measure zero. By Theorem \ref{Theorem: Existence of SILT}, this same property is preserved in the limits
$\ga_A$ and $\be_A$, with an almost sure equality.
In particular, since the sets in the union
\[[0,t]_\leq^2\cup[t,2t]_{\leq}\cup\big([0,t]\times[t,2t]\big)=[0,2t]^2_{\leq}\]
only intersect on a set of measure zero, it follows that
\begin{align}
\label{Equation: HS Expansion GOAL Averaged 1.2}
\text{right-hand side of }\eqref{Equation: HS Expansion GOAL Averaged 1.1}=\ga_{2t}(B^{x,x}_{2t})\qquad\text{almost surely.}
\end{align}
On the other hand, by \eqref{Equation: Renormalization} and Lemma \ref{Lemma: Smoothed Expectation of SILT},
\begin{multline}
\label{Equation: HS Expansion GOAL Averaged 2.1}
\lim_{\eps,\tilde\eps\to0}\eqref{Equation: HS Expansion GOAL Averaged 2}=2\frac{\ka^2}{2\pi}\left(t\log(2t)+\int_0^t\log r-\log(2t-r)\d r\right)\\
+\frac{\ka^2}{2\pi}\left(\int_{t}^{2t}\log r-\log(2t-r)\d r
-\int_{0}^{t}\log r-\log(2t-r)\d r\right)=\frac{\ka^2t\log(2t)}\pi,
\end{multline}
where the last equality follows from the fact that the three integrals above cancel, since the map $f(r)=\log r-\log(2t-r)$
satisfies $f(r)=-f(2t-r)$ for $r\in[0,t]$.

Now that we have proved \eqref{Equation: HS Expansion GOAL Averaged 1.1}, \eqref{Equation: HS Expansion GOAL Averaged 1.2},
and \eqref{Equation: HS Expansion GOAL Averaged 2.1}, we see that \eqref{Equation: HS Expansion GOAL Averaged}
follows if we show that we can take the limits $\eps,\tilde\eps\to0$ inside the $\dd x$ integral and expectation therein.
For this purpose, if we combine the trivial bound $\mbf 1_{\{\cdot\}}\leq1$ with H\"older's inequality,
it suffices to show that for any $\eta\geq1$,
for small enough $t$ it holds that
\begin{multline}
\label{Equation: HS Expansion GOAL UI}
\sup_{x\in D,~\eps,\tilde\eps>0}
\mbf E\left[\mr e^{3\eta(\ka^2\be^\eps_{[0,t]^2_{\leq}}(B^{x,x}_{2t})-t\msf c_{\ka,\eps})}\right]^{1/3}
\mbf E\left[\mr e^{3\eta(\ka^2\be^{\tilde\eps}_{[t,2t]^2_{\leq}}(B^{x,x}_{2t})-t\msf c_{\ka,\tilde\eps})}\right]^{1/3}\\
\mbf E\left[\mr e^{3\eta\ka^2\be^{(\eps+\tilde\eps)/2}_{[0,t]\times[t,2t]}(B^{x,x}_{2t})}\right]^{1/3}<\infty.
\end{multline}
Indeed, this simultaneously shows that we can apply the dominated convergence theorem in the $\dd x$
integral (since $D$ is bounded) and the Vitali convergence theorem in the expectation
(since \eqref{Equation: HS Expansion GOAL UI} for $\eta>1$ shows that the random variables
in question are uniformly integrable by \eqref{Equation: Poussin}).
The finiteness of the suprema on the first line of
\eqref{Equation: HS Expansion GOAL UI} can be proved in the
same way as the finiteness of the suprema on the first line of \eqref{Equation: UI after Holder};
the finiteness of the second line of \eqref{Equation: HS Expansion GOAL UI} follows from
\eqref{Equation: Uniform Integrability of SILT - Non-Renormalized}.

\section{Theorem \ref{Theorem: Main} Step 3: Intersection Local Times Asymptotics}
\label{Section: Intersection Local Times Asymptotics}

We now use the Feynman-Kac formulas in Proposition \ref{Proposition: Feynman-Kac for T and M}
to prove the asymptotics in Theorem \ref{Theorem: Main}. Given that their proofs
differ slightly, we deal with the expectation
and variance asymptotics separately.

\subsection{Expectation Asymptotics}

By the classical Feynman-Kac formula,
\[\msf T_0(t)=\frac{1}{2\pi t}\int_D\mbf E\Big[\mbf 1_{\{\tau_D(B^{x,x}_t)>t\}}\Big]\d x
\qquad\text{and}\qquad
\msf M_0(t)=\int_D\mbf E\Big[\mbf 1_{\{\tau_D(B^x)>t\}}\Big]\d x.\]
If we combine this with Propositions \ref{Proposition: Moments Prelimit} and \ref{Proposition: Feynman-Kac for T and M},
then the expectation asymptotics in \eqref{Equation: Main Asymptotic 1}
and \eqref{Equation: Main Asymptotic 2} can be respectively reduced to
\begin{multline}
\label{Equation: Expectation Asymptotics 1}
\frac{\mr e^{\ka^2t\log t/2\pi}}{2\pi t}\int_D\mbf E\Big[\mbf 1_{\{\tau_D(B^{x,x}_t)>t\}}\mr e^{\ka^2\ga_t(B^{x,x}_t)}\Big]\d x\\
=\frac{1}{2\pi t}\int_D\mbf E\Big[\mbf 1_{\{\tau_D(B^{x,x}_t)>t\}}\Big]\d x+\tfrac{\ka^2\msf{A}(D)}{4\pi^2}\log t+o(\log t)
\qquad\text{as }t\to0,
\end{multline}
and similarly
\begin{multline}
\label{Equation: Expectation Asymptotics 2}
\mr e^{\ka^2(t\log t-t)/2\pi}\int_D\mbf E\Big[\mbf 1_{\{\tau_D(B^x)>t\}}\mr e^{\ka^2\ga_t(B^x)}\Big]\d x\\
=\int_D\mbf E\Big[\mbf 1_{\{\tau_D(B^x)>t\}}\Big]\d x+\tfrac{\ka^2\msf{A}(D)}{2\pi}t\log t+o(t\log t)\qquad\text{as }t\to0.
\end{multline}

Let us begin with \eqref{Equation: Expectation Asymptotics 1}.
By a Taylor expansion,
\begin{align}
\label{Equation: Expectation Asymptotics Deterministic Taylor}
\mr e^{\ka^2t\log t/2\pi}=1+\frac{\ka^2}{2\pi}t\log t+O(t^2\log^2t)\qquad\text{as } t\to0.
\end{align}
As $t^2\log^2t=o(t\log t)$, the left-hand side of \eqref{Equation: Expectation Asymptotics 1} can be written as
\begin{multline}
\label{Equation: Expectations Asymptotics 1.1}
\frac{1}{2\pi t}\int_D\mbf E\Big[\mbf 1_{\{\tau_D(B^{x,x}_t)>t\}}\mr e^{\ka^2\ga_t(B^{x,x}_t)}\Big]\d x\\
+\frac{1}{2\pi t}\left(\frac{\ka^2}{2\pi}t\log t+o(t\log t)\right)\int_D\mbf E\Big[\mbf 1_{\{\tau_D(B^{x,x}_t)>t\}}\mr e^{\ka^2\ga_t(B^{x,x}_t)}\Big]\d x.
\end{multline}
With this in hand, \eqref{Equation: Expectation Asymptotics 1} is now a consequence of the claims that
\begin{align}
\label{Equation: Exp of SILT is Negligible}
\int_D\mbf E\Big[\mbf 1_{\{\tau_D(B^{x,x}_t)>t\}}\mr e^{\ka^2\ga_t(B^{x,x}_t)}\Big]\d x
=\int_D\mbf E\Big[\mbf 1_{\{\tau_D(B^{x,x}_t)>t\}}\Big]\d x+o(t)\qquad\text{as }t\to0,
\end{align}
and
\begin{align}
\label{Equation: First Order of T0}
\int_D\mbf E\Big[\mbf 1_{\{\tau_D(B^{x,x}_t)>t\}}\Big]\d x=\msf A(D)+o(1)\qquad\text{as }t\to0.
\end{align}
(More specifically, applying \eqref{Equation: Exp of SILT is Negligible}
in the first line of \eqref{Equation: Expectations Asymptotics 1.1},
and then successively applying \eqref{Equation: Exp of SILT is Negligible}
and \eqref{Equation: First Order of T0}
in the second line of \eqref{Equation: Expectations Asymptotics 1.1}.)

We begin with \eqref{Equation: First Order of T0}, since one element of its
proof is also used in \eqref{Equation: Exp of SILT is Negligible}.
Given that
\begin{align}
\label{Equation: 1A=1-1Ac trick}
\mbf 1_{\{\tau_D(B^{x,x}_t)>t\}}=1-\mbf 1_{\{\tau_D(B^{x,x}_t)\leq t\}},
\end{align}
it suffices to show that for every $0<\theta\leq1$,
\begin{align}
\label{Equation: First Order of T0 2}
\int_D\mbf E\Big[\mbf 1_{\{\tau_D(B^{x,x}_t)\leq t\}}\Big]^{\theta}\d x=o(1)\qquad\text{as }t\to0
\end{align}
(we only need the result with $\theta=1$ for \eqref{Equation: First Order of T0},
but we will need the case of $\theta<1$ in the proof of \eqref{Equation: Exp of SILT is Negligible}).
By Brownian scaling, we can couple the Brownian bridges appearing in \eqref{Equation: First Order of T0 2}
for different $x$'s and $t$'s in such a way that
\[B^{x,x}_t=x+\sqrt{t}\,B^{0,0}_1,\qquad t>0,~x\in D.\]
Under this coupling, we observe that for every $x\in D$,
\[\lim_{t\to0}\mbf 1_{\{\tau_D(B^{x,x}_t)\leq t\}}=\lim_{t\to0}\mbf 1_{\{\tau_D(x+\sqrt{t}B^{0,0}_1)\leq 1\}}=0\qquad\text{almost surely}.\]
Indeed, since $D$ is open, for any point $x\in D$, there exists some $\eps_x>0$ small enough so that the ball
$\mc B(x,\eps_x)=\{y\in\mbb R^2:|x-y|<\eps\}$ is contained in $D$. Since
\[\mc S=\sup_{s\in[0,1]}|B^{0,0}_1(s)|<\infty\qquad\text{almost surely},\]
we have that $x+\sqrt{t}B^{0,0}_1(s)\in \mc B(x,\eps_x)$ for all $s\in[0,1]$ (hence $\mbf 1_{\{\tau_D(x+\sqrt{t}B^{0,0}_1)\leq 1\}}=0$)
whenever $\sqrt t\mc S<\eps_x$, which is equivalent to $t<(\eps_x/\mc S)^2$.
We then get \eqref{Equation: First Order of T0 2} by dominated convergence
thanks to the trivial bound $\mbf 1_{\{\cdot\}}\leq1$.

We now move on to \eqref{Equation: Exp of SILT is Negligible}.
By a Taylor expansion, we can write
\[\mr e^z=1+z+\msf R(z),\qquad\msf |R(z)|
\leq z^2(1+\mr e^z),\qquad z\in\mbb R.\]
If we apply this to $z=\ka^2\ga_t(B^{x,x}_t)$, then
\eqref{Equation: Exp of SILT is Negligible} can be reduced to
\begin{align}
\label{Equation: Bridge Expectation Asymptotics 1}
&\mbf E\left[\int_D\mbf 1_{\{\tau_D(B^{x,x}_t)>t\}}\ga_t(B^{x,x}_t)\d x\right]=o(t)\qquad\text{as }t\to0,\\
\label{Equation: Bridge Expectation Asymptotics 2}
&\mbf E\left[\int_D\mbf 1_{\{\tau_D(B^{x,x}_t)>t\}}\ga_t(B^{x,x}_t)^2(1+\mr e^{\ka^2\ga_t(B^{x,x}_t)})\d x\right]=o(t)\qquad\text{as }t\to0.
\end{align}
Using \eqref{Equation: 1A=1-1Ac trick} and Tonelli's theorem, we get that
\[\text{left-hand side of \eqref{Equation: Bridge Expectation Asymptotics 1}}=\int_D\mbf E\left[\ga_t(B^{x,x}_t)\right]\d x
-\int_D\mbf E\left[\mbf 1_{\{\tau_D(B^{x,x}_t)\leq t\}}\ga_t(B^{x,x}_t)\right]\d x.\]
On the one hand, given that $\ga_t(B^{x,x}_t)$ is constructed as the limit of random variables with expectation zero
(Theorem \ref{Theorem: Existence of SILT}-(2)) that are uniformly integrable (i.e., \eqref{Equation: Uniform Integrability of SILT - Renormalized}), we get that $\int_D\mbf E\left[\ga_t(B^{x,x}_t)\right]\dd x=0.$
On the other hand, H\"older's inequality
and the scaling property in Lemma \ref{Lemma: Scaling of Approximate SILT} implies that
\[\int_D\mbf E\left[\mbf 1_{\{\tau_D(B^{x,x}_t)\leq t\}}\ga_t(B^{x,x}_t)\right]\d x
\leq t\cdot\mbf E\left[\ga_1(B^{0,0}_1)^2\right]^{1/2}\int_D\mbf E\left[\mbf 1_{\{\tau_D(B^{x,x}_t)\leq t\}}\right]^{1/2}\d x.\]
Since the moments of $\ga_1(B^{0,0}_1)$ are finite (i.e., \eqref{Equation: Exponential Moment SILT}),
this term is on the order of $o(t)$ by \eqref{Equation: First Order of T0 2} with $\theta=1/2$;
hence \eqref{Equation: Bridge Expectation Asymptotics 1} holds.

We now move on to \eqref{Equation: Bridge Expectation Asymptotics 2}.
If we combine the trivial bound $\mbf 1_{\{\cdot\}}\leq1$
with H\"older's inequality and the scaling property in Lemma \ref{Lemma: Scaling of Approximate SILT}, then we get
\[\text{left-hand side of }\eqref{Equation: Bridge Expectation Asymptotics 2}\leq
\msf A(D)\,t^2\,\mbf E\left[\ga_1(B^{0,0}_1)^4\right]^{1/2}\mbf E\left[\left(1+\mr e^{\ka^2t\ga_1(B^{0,0}_1)}\right)^2\right]^{1/2}.\]
This decays on the order of $O(t^2)=o(t)$ as $t\to0$ by \eqref{Equation: Uniform Integrability of SILT - Renormalized},
thus proving \eqref{Equation: Bridge Expectation Asymptotics 2}.

With this done, we have now concluded the proof of \eqref{Equation: Expectation Asymptotics 1},
and thus of the expectation asymptotic for $\mbf E\big[\msf T_\ka(t)\big]$ in Theorem \ref{Theorem: Main}.
In order to prove the corresponding asymptotic for $\mbf E\big[\msf M_\ka(t)\big]$, i.e.,
\eqref{Equation: Expectation Asymptotics 2}, we can go through exactly the same steps
with every instance of $B^{x,x}_t$ replaced by $B^x$,
but with the following minor differences:
\begin{enumerate}[(1)]
\item The Taylor expansion \eqref{Equation: Expectation Asymptotics Deterministic Taylor}
is replaced by
\[\mr e^{\ka^2(t\log t-t)/2\pi}=1+\frac{\ka^2}{2\pi}t\log t+O(t)\]
due to the presence of $-\ka^2t/2\pi$.
\item The analogue of \eqref{Equation: Expectations Asymptotics 1.1} for $\msf M_\ka(t)$
has no factor of the form $\frac1{2\pi t}$.
\end{enumerate}
This concludes the proof of the expectation asymptotics in Theorem \ref{Theorem: Main}.

\subsection{Variance Asymptotics}

We begin with the variance bound for $\msf T_\ka(t)$.
Since $\mbf 1_{\{\cdot\}}\leq 1$ and $\frac{\mr e^{\ka^2t\log t/\pi}}{2\pi}=O(1)$ as $t\to0$,
an application of H\"older's inequality in \eqref{Equation: Feynman-Kac for T and M 3} (together with Proposition \ref{Proposition: Moments Prelimit}) yields the $t\to0$ asymptotic
\[\mbf{Var}\big[\msf T_\ka(t)\big]=O\left(t^{-2}\int_{D^2}\mbf E\left[\mr e^{2\ka^2\sum_{i=1}^2\ga_t({_i}B^{x_i,x_i}_t)}\right]^{1/2}
\mbf E\left[\left(\mr e^{\ka^2\al_t({_1}B^{x_1,x_1}_t,{_2}B^{x_2,x_2}_t)}-1\right)^2\right]^{1/2}\Bigg]\d x\right).\]
By \eqref{Equation: Exponential Moment SILT}, this yields
\[\mbf{Var}\big[\msf T_\ka(t)\big]=O\left(t^{-2}\int_{D^2}
\mbf E\left[\left(\mr e^{\ka^2\al_t({_1}B^{x_1,x_1}_t,{_2}B^{x_2,x_2}_t)}-1\right)^2\right]^{1/2}\Bigg]\d x\right).\]
Given that $\mr e^z-1\leq|z|\mr e^{|z|}$, another application of H\"older's inequality yields
\[\mbf{Var}\big[\msf T_\ka(t)\big]=O\left(t^{-2}\int_{D^2}
\mbf E\left[\al_t({_1}B^{x_1,x_1}_t,{_2}B^{x_2,x_2}_t)^4\right]^{1/4}
\mbf E\left[\mr e^{4\ka^2\al_t({_1}B^{x_1,x_1}_t,{_2}B^{x_2,x_2}_t)}\right]^{1/4}\d x\right).\]
By Lemma \ref{Lemma: MILT is UI}, the term involving an exponential above is uniformly bounded over
$x$ for small $t$, hence
\[\mbf{Var}\big[\msf T_\ka(t)\big]=O\left(t^{-2}\int_{D^2}
\mbf E\left[\al_t({_1}B^{x_1,x_1}_t,{_2}B^{x_2,x_2}_t)^4\right]^{1/4}\d x\right).\]
Therefore,
we obtain the variance bound in \eqref{Equation: Main Asymptotic 1}
if we show that, as $t\to0$,
\begin{align}
\label{Equation: Variance to MILT}
\int_{D^2}
\mbf E\left[\al_t({_1}B^{x_1,x_1}_t,{_2}B^{x_2,x_2}_t)^4\right]^{1/4}\d x=O(t^2).
\end{align}

By Lemma \ref{Lemma: No Intersection Means no MILT},
for any $x=(x_1,x_2)\in D^2$,
we can write
\[\mbf E\left[\al_t({_1}B^{x_1,x_1}_t,{_2}B^{x_2,x_2}_t)^4\right]^{1/4}
=\mbf E\left[\al_t({_1}B^{x_1,x_1}_t,{_2}B^{x_2,x_2}_t)^4\mbf 1_{\mc A_1(t,x)\cup\mc A_2(t,x)}\right]^{1/4},\]
where we define the events
\[\mc A_i(t,x)=\left\{\sup_{r\in[0,t]}\left|{_i}B^{x_i,x_i}_t(r)-x_i\right|\geq\frac{|x_1-x_2|}{3}\right\},\qquad i=1,2.\]
Indeed, if $\mc A_i(t,x)^c\cap\mc A_i(t,x)^c$ holds, then the paths of ${_i}B^{x_i,x_i}_t$ are separated
by at least $\theta=|x_1-x_2|/3$, hence the event $\mf O_{t,\theta}({_1}B^{x_1,x_1}_t,{_2}B^{x_2,x_2}_t)$ holds.
Thus, by H\"older's inequality, we have the estimate
\begin{multline*}
\int_{D^2}
\mbf E\left[\al_t({_1}B^{x_1,x_1}_t,{_2}B^{x_2,x_2}_t)^4\right]^{1/4}\d x\\
\leq
\int_{D^2}
\mbf E\left[\al_t({_1}B^{x_1,x_1}_t,{_2}B^{x_2,x_2}_t)^8\right]^{1/8}\mbf P\big[\mc A_1(t,x)\cup\mc A_2(t,x)\big]^{1/8}\d x.
\end{multline*}
By Lemmas \ref{Lemma: MILT Scaling Property} and \ref{Lemma: MILT is UI},
\[\sup_{x\in D^2}\mbf E\left[\al_t({_1}B^{x_1,x_1}_t,{_2}B^{x_2,x_2}_t)^8\right]^{1/8}=O(t).\]
Therefore, we get that
\[\int_{D^2}
\mbf E\left[\al_t({_1}B^{x_1,x_1}_t,{_2}B^{x_2,x_2}_t)^4\right]^{1/4}\d x=
O\left(t\int_{D^2}\mbf P\big[\mc A_1(t,x)\cup\mc A_2(t,x)\big]^{1/8}\d x\right).\]
If we combine a Brownian scaling with the fact that
Brownian motion and bridge suprema have Gaussian tails, then we conclude that
there exist some
constants $C,c>0$ such that for $i=1,2$, one has
\[\mbf P\big[\mc A_i(t,x)\big]\leq\mbf P\left[\sup_{r\in[0,1]}\left|{_i}B^{0,0}_1(r)\right|\geq\frac{|x_1-x_2|}{3t^{1/2}}\right]
\leq C\mr e^{-c|x_1-x_2|^2/t}.\]
Thus, by a union bound,
\[\int_{D^2}
\mbf E\left[\al_t({_1}B^{x_1,x_1}_t,{_2}B^{x_2,x_2}_t)^4\right]^{1/4}\d x=
O\left(t\int_{D^2}\mr e^{-c|x_1-x_2|^2/8t}\d x\right),\]
which yields \eqref{Equation: Variance to MILT} since $\int_{D^2}\mr e^{-c|x_1-x_2|^2/8t}\dd x=O(t)$.

We have now proved the variance bound for $\msf T_\ka(t)$ in Theorem \ref{Theorem: Main}.
In order to get the corresponding asymptotic for $\msf M_\ka(t)$, we apply exactly the
same argument with Brownian bridges replaced by Brownian motions, the only
meaningful difference being that the absence of the factor $\frac{1}{4\pi t^2}$ in \eqref{Equation: Feynman-Kac for T and M 4}
yields a variance bound on the order of $O(t^2)$ instead.
The proof of Theorem \ref{Theorem: Main} is thus complete.

\section{Proof of Applications}
\label{Section: Applications}

\subsection{Proof of Corollaries \ref{Corollary: Spectral Geometry} and \ref{Corollary: Variance}}
\label{Section: Proof of Spectral Geometry}

We first consider Corollary \ref{Corollary: Spectral Geometry}-(1).
By \eqref{Equation: Main Asymptotic 1} and \eqref{Equation: Smooth Heat Trace},
\[\mbf E\big[2\pi t_n\msf T_\ka(t_n)\big]=\msf{A}(D)+o(1)
\qquad\text{and}\qquad
\mbf{Var}\big[2\pi t_n\msf T_\ka(t_n)\big]=O(t_n^2)\]
as $n\to\infty$. Thus, by a combination of the Borel-Cantelli lemma
and Chebyshev's inequality, it suffices to show that there exists some $\eta\in(0,2)$ such that
\[\mbf P\left[\Big|2\pi t_n\msf T_\ka(t_n)-\mbf E\big[2\pi t_n\msf T_\ka(t_n)\big]\Big|>t_n^{\eta/2}\right]=O(t_n^{2-\eta})\]
is summable in $n$. If $t_n\leq cn^{-1/2-\eps}$ for some $c,\eps>0$, then we get that
\[\sum_{n=1}^\infty t_n^{2-\eta}\leq c\sum_{n=1}^\infty n^{(-1/2-\eps)(2-\eta)}
=c\sum_{n=1}^\infty n^{-1-2\eps+\eta(1/2+\eps)}.\]
For every $\eps>0$, we can always find $\eta$ small enough so that
$-2\eps+\eta(1/2+\eps)<0$, thus making the above summable.

Corollaries \ref{Corollary: Spectral Geometry}-(2) and \ref{Corollary: Variance}
are proved similarly:
By \eqref{Equation: Main Asymptotic 1} and \eqref{Equation: Smooth Heat Trace},
\[\mbf E\left[4\sqrt{2\pi t_n}\left(\tfrac{\msf{A}(D)}{2\pi}t_n^{-1}-\msf T_\ka(t_n)\right)\right]=\msf{L}(\partial D)+o(1),\]
\[\mbf{Var}\left[4\sqrt{2\pi t_n}\left(\tfrac{\msf{A}(D)}{2\pi}t_n^{-1}-\msf T_\ka(t_n)\right)\right]=O(t_n),\]
as $n\to\infty$, and similarly by \eqref{Equation: Main Asymptotic 1},
\[\mbf E\left[\tfrac{4\pi^2}{\msf{A}(D)\log t_n}\big(\mathsf T_\ka(t_n)-\msf T_0(t_n)\big)\right]=\ka^2+o(1),\]
\[\mbf{Var}\left[\tfrac{4\pi^2}{\msf{A}(D)\log t_n}\big(\mathsf T_\ka(t_n)-\msf T_0(t_n)\big)\right]=O\big((\log t_n)^{-2}\big).\]
Thus, by Borel-Cantelli and Chebyshev's inequality, it suffices to show that there exists $\eta>0$
such that $\sum_{n=1}^\infty t_n^{1-\eta}<\infty$ in the case of Corollary \ref{Corollary: Spectral Geometry}-(2)
(this works if $t_n\leq cn^{-1-\eps}$),
and similarly $\sum_{n=1}^\infty |\log t_n|^{-2+\eta}<\infty$ in the case of Corollary \ref{Corollary: Variance}
(this works if $t_n\leq\tilde c\mr e^{-cn^{1/2+\eps}}$). This concludes the proofs of Corollaries \ref{Corollary: Spectral Geometry}
and \ref{Corollary: Variance}.

\subsection{Proof of Corollary \ref{Corollary: Fractal Dimension}}
\label{Section: Proof of Fractal Dimension}

It suffices to show that there exists random variables $\mf c,\mf N>1$ such that,
almost surely,
\begin{align}
\label{Equation: Fractal Dimension Sufficient Condition}
\mf c^{-1}\,t_n^{1-\msf d_M(\partial D)/2}\leq\msf{A}(D)-\msf M_\ka(t_n)\leq\mf c\,t_n^{1-\msf d_M(\partial D)/2}
\qquad\text{for every }n\geq\mf N.
\end{align}
For this purpose, let us write
\[\msf{A}(D)-\msf M_\ka(t_n)
=\big(\msf{A}(D)-\msf M_0(t_n)\big)
+\big(\msf M_0(t_n)-\mbf E\big[\msf M_\ka(t_n)\big]\big)
+\big(\mbf E\big[\msf M_\ka(t_n)\big]-\msf M_\ka(t_n)\big).\]
Since $t\log t=o(t^{1-\msf d_M(\partial D)/2})$,
thanks to the expectation asymptotic in \eqref{Equation: Main Asymptotic 2} and \eqref{Equation: Fractal Dimension Heat Content},
there exists nonrandom constants $\tilde c,\tilde N>1$ such that
\[\tilde c^{-1}\,t_n^{1-\msf d_M(\partial D)/2}\leq\big(\msf{A}(D)-\msf M_0(t_n)\big)+\big(\msf M_0(t_n)-\mbf E\big[\msf M_\ka(t_n)\big]\big)\leq\tilde c\,t_n^{1-\msf d_M(\partial D)/2}
\qquad\text{for }n\geq\tilde N.\]
Thus, it suffices to show that for every $\eta>0$ (in particular, $\eta<\tilde c$),
\[\mbf P\left[\big|\msf M_\ka(t_n)-\mbf E\big[\msf M_\ka(t_n)\big]\big|>\eta t_n^{1-\msf d_M(\partial D)/2}\text{ infinitely often}\right]=0.\]
By the Borel-Cantelli lemma and Chebyshev's inequality (using the variance bound in \eqref{Equation: Main Asymptotic 2}), this follows from
\[\sum_{n=1}^\infty t_n^{2-2(1-\msf d_M(\partial D)/2)}=\sum_{n=1}^\infty t_n^{\msf d_M(\partial D)}<\infty.\]
If we choose $t_n\leq cn^{-1/\msf d_M(\partial D)-\eps}$ for any fixed $c,\eps>0$,
then we get
\[\sum_{n=1}^\infty t_n^{\msf d_M(\partial D)}\leq c\sum_{n=1}^\infty n^{-1-\eps\msf d_M(\partial D)}<\infty,\]
as desired.

\bibliographystyle{plain}
\bibliography{Bibliography}

\end{document}